%% file: totalLG.tex
\documentclass[a4paper,11pt,leqno]{amsart}
\usepackage[applemac]{inputenc}
\usepackage{epsfig,graphics}
\usepackage[all,knot]{xy}
\usepackage{amsthm,amssymb,amsbsy,bbm,a4wide,verbatim,url}
\usepackage{longtable}
\usepackage{rotating}
\xyoption{knot}

\def\kk{\mathbbm{k}}
\def\Z{\mathbbm{Z}}
\def\C{\mathbbm{C}}
\def\Q{\mathbbm{Q}}
\def\un{\mathbbm{1}}
\def\tr{\mathrm{tr}}
\def\End{\mathrm{End}}
\def\GL{\mathrm{GL}}
\def\Ker{\mathrm{Ker}}
\def\Sp{\mathrm{Sp}}
\def\rk{\mathrm{rk}}
\def\Id{\mathrm{Id}}
\def\Hom{\mathrm{Hom}}
\def\Res{\mathrm{Res}}
\def\g{\mathfrak{g}}
\def\h{\mathfrak{h}}
\def\la{\lambda}
\def\gl{\mathfrak{gl}}
\def\sl{\mathfrak{sl}}
\def\so{\mathfrak{so}}
\def\sp{\mathfrak{sp}}
\def\osp{\mathfrak{osp}}
\def\U{\mathsf{U}}
\def\eps{\varepsilon}
\def\into{\hookrightarrow}
\def\onto{\twoheadrightarrow}

\newtheorem{theor}[equation]{Theorem}
\newtheorem{lemma}[equation]{Lemma}
\newtheorem{prop}[equation]{Proposition}
\newtheorem{cor}[equation]{Corollary}
\newtheorem{conj}[equation]{Conjecture}
\newtheorem{question}[equation]{Question}
\newtheorem{remark}[equation]{Remark}
\numberwithin{equation}{section}

\usepackage{tikz}
\newcounter{braid}
\newcounter{strands}
\pgfkeyssetvalue{/tikz/braid height}{1cm}
\pgfkeyssetvalue{/tikz/braid width}{1cm}
\pgfkeyssetvalue{/tikz/braid start}{(0,0)}
\pgfkeyssetvalue{/tikz/braid colour}{black}
\pgfkeys{/tikz/strands/.code={\setcounter{strands}{#1}}}

\makeatletter
\def\cross{%
  \@ifnextchar^{\message{Got sup}\cross@sup}{\cross@sub}}

\def\cross@sup^#1_#2{\render@cross{#2}{#1}}

\def\cross@sub_#1{\@ifnextchar^{\cross@@sub{#1}}{\render@cross{#1}{1}}}

\def\cross@@sub#1^#2{\render@cross{#1}{#2}}

\def\render@cross#1#2{
  \def\strand{#1}
  \def\crossing{#2}
  \pgfmathsetmacro{\cross@y}{-\value{braid}*\braid@h}
  \pgfmathtruncatemacro{\nextstrand}{#1+1}
  \foreach \thread in {1,...,\value{strands}}
  {
    \pgfmathsetmacro{\strand@x}{\thread * \braid@w}
    \ifnum\thread=\strand
    \pgfmathsetmacro{\over@x}{\strand * \braid@w + .5*(1 - \crossing) * \braid@w}
    \pgfmathsetmacro{\under@x}{\strand * \braid@w + .5*(1 + \crossing) * \braid@w}
    \draw[braid] \pgfkeysvalueof{/tikz/braid start} +(\under@x pt,\cross@y pt) to[out=-90,in=90] +(\over@x pt,\cross@y pt -\braid@h);
    \draw[braid] \pgfkeysvalueof{/tikz/braid start} +(\over@x pt,\cross@y pt) to[out=-90,in=90] +(\under@x pt,\cross@y pt -\braid@h);
    \else
    \ifnum\thread=\nextstrand
    \else
     \draw[braid] \pgfkeysvalueof{/tikz/braid start} ++(\strand@x pt,\cross@y pt) -- ++(0,-\braid@h);
    \fi
   \fi
  }
  \stepcounter{braid}
}

\tikzset{braid/.style={double=\pgfkeysvalueof{/tikz/braid colour},double distance=1pt,line width=2pt,white}}

\newcommand{\braid}[2][]{%
  \begingroup
  \pgfkeys{/tikz/strands=2}
  \tikzset{#1}
  \pgfkeysgetvalue{/tikz/braid width}{\braid@w}
  \pgfkeysgetvalue{/tikz/braid height}{\braid@h}
  \setcounter{braid}{0}
  \let\dsigma=\cross
  #2
  \endgroup
}
\makeatother

\title[Defining algebra for the Links-Gould polynomial]{A cubic defining algebra for the Links-Gould polynomial}

\author{{\large I\lowercase{van} M\lowercase{arin} \& E\lowercase{mmanuel} W\lowercase{agner}} }

\begin{document}

\setcounter{tocdepth}{1}
\maketitle

\begin{center}
Institut de Math\'ematiques de Jussieu \\
Universit\'e Paris 7 \\
175 rue du Chevaleret \\
75013 Paris \\
France\\
\url{marin@math.jussieu.fr}\\
--- \\
Institut de Math\'ematiques de Bourgogne UMR 5584 \\
Universit\'e de Bourgogne\\
7 Avenue Alain Savary BP 47870\\
21078 Dijon Cedex\\
France\\
\url{emmanuel.wagner@u-bourgogne.fr}
\end{center}
\bigskip

\tableofcontents

\section{Introduction}



In 1992, Links and Gould introduced a polynomial invariant of knots
and links out of a family of 4-dimensional representations of
the quantum Lie superalgebra $\mathsf{U}_q \sl(2|1)$. This invariant satisfies
a \emph{cubic} skein relation, that is the simplest skein relation
on simple crossings that can be asked for, the quadratic one being
characteristic of the HOMFLY-PT polynomial. It shares this property with
the Kauffman polynomial (which corresponds to the quantum orthosymplectic
Lie algebras and their standard representations), but behaves quite differently,
notably with respect to disjoint union of links : the Kauffman polynomial
is multiplicative with respect to the disjoint union of two links, whereas the Links-Gould
polynomial vanishes on such disjoint unions.

The additional skein relations satisfied by the Kauffman polynomial are relations
of the so-called Birman-Wenzl-Murakami (BMW) algebra (see \cite{BIRMANWENZL}, \cite{MURAKAMI}). This $BMW_n$
algebra is a quotient of the group algebra $K B_n$ of the braid group
over some field $K$ of characteristic $0$
by a generic cubic relations and by some other relation in $K B_3$, and there
exists a single Markov trace on the tower of algebras $(BMW)_n$, whose value
on closed braids provides the Kauffman invariant. This algebra is a deformation
of the classical algebra of Brauer diagrams, and as such admits a basis
with a nice combinatorial description. It describes the centralizer algebra of the
action of $\mathsf{U}_q \mathfrak{osp}(V)$ inside $V^{\otimes n}$.
In addition, since the faithful Krammer representation
factorizes through this  $BMW$ algebra, $B_n$ embeds into the group $BMW_n^{\times}$
of invertible elements of the $BMW$ algebra ; its Zariski closure is described in \cite{BMW}.

The goal of this paper is to define a similar algebra for the Links-Gould polynomial.
We first consider the corresponding centralizer algebra $LG_n$ and prove the
following statement, analogous to the well-known fact that the $BMW_n$ algebra, defined as
a (quantum) centralizer algebra, is a quotient of $K B_n$ (see corollary \ref{corsurj}).

\begin{theor} The natural morphism $K B_n \to LG_n$ is surjective.
\end{theor}

As a consequence, $LG_n$ is a natural candidate for being an analogue of
the $BMW_n$ algebra for the Links-Gould polynomial. As it is a centralizer
algebra, we have a natural (combinatorial) description of its simple modules, but
we do not have yet a satisfactory description of its elements. We have the
following conjecture about its dimension.

\begin{conj} For all $n$,
$$
\dim LG_{n+1} = \frac{(2n)!(2n+1)!}{(n! (n+1)!)^2}
$$
\end{conj}
We checked this formula for $n \leq 50$.  It gives
$\dim LG_1 = 1$, $\dim LG_2 = 3$, $\dim LG_3 = 20$, $\dim LG_4 = 175$, $\dim LG_5 = 1764$, $\dim LG_6 = 19404$.
As communicated to us by F. Chapoton, this formula appears in the following
setting. It is the number of pairs of paths, inside a square whose side has length $n+1$,
which go from the top left corner to the down right by down and right moves, which do not cross each other (see figure \ref{figchemsLG3}
for the corresponding 20 diagrams for $LG_3$). It suggests
the possibility of a natural basis of this algebra indexed by such combinatorial objects.
Sloane's encyclopedia of integer sequences suggests another (related)
combinatorial interpretation, as the number of non-crossing partitions
of $2n+1$ inside $n+1$ blocs. R. Blacher suggested to us to consider it as $(C_n)^2 (2n+1)$, where $C_n$
is the Catalan number, and to interpret it as couples of binary trees together with a suitable marking of the
leaves. In any case, it is natural to ask
for a pictorial description of this algebra, and also if it admits
a natural `cellular structure' in the sense of Graham and Lehrer.

Our next goal is to get a presentation of $LG_n$ by generators and relations. We know
that it is a quotient of the generic cubic Hecke algebra, namely the quotient $H_n$ of
the group algebra of the braid group by a generic cubic relation. The latter is an infinite dimensional
algebra for $n \geq 6$ (see section \ref{secthecke}).
In \cite{Ishii3}, Ishii introduced a relation $r_2 \in H_3$ satisfied inside the skein
module of the Links-Gould polynomial, and proved that
the quotient of $H_3$ by (the image of) $r_2$ has dimension $20$. We first prove that
the natural morphism $H_4/(r_2) \onto LG_4$ is \emph{not} an isomorphism, and introduce
a new relation $(r_3)$ such that $H_4/(r_3) = H_4/(r_2,r_3) = LG_4$.
We let $A_n = H_n/(r_2,r_3)$ for $n \geq 4$, $A_3 = H_3/(r_2)$, $A_2 = H_2$. By definition,
$A_n \simeq LG_n$ for $n \leq 4$. We prove (see section \ref{Markovtrace}) the following.
\begin{theor} \ 
\begin{enumerate}
\item $A_5 \simeq LG_5$.
\item $A_n$ is finite dimensional for all $n$.
\item There exists only one Markov trace on $(A_n)_{n \geq 1}$.
\end{enumerate}
\end{theor}

We also get that there is only one Markov trace on $LG_n$. These two algebras
$A_n$ and $LG_n$ are related by $A_n \onto LG_n$ and the above theorem provides
some evidence in favor of the following conjecture.

\begin{conj} \label{conjiso} For all $n$, $A_n \simeq LG_n$. 

\end{conj}

The conjunction of these two conjectures clearly implies the next one :

\begin{conj} \label{conjdimAn} For all $n$,
$$
\dim A_{n+1} = \frac{(2n)!(2n+1)!}{(n! (n+1)!)^2}.
$$
\end{conj}

While this conjecture is open, we have two (conjecturally equal) algebras, and
one may use either of them, depending on our needs, to deal with
the Links-Gould polynomial.

A computational use of the new relation $r_3$
(that is, of the algebra $A_n$) is theoretically possible, as it provides
a new algorithm, which should use less memory space than the
brutal use of the R-matrix. Unfortunately, the relation $r_3$
has too many terms to be printed here, not to mention to be used by a human.
A computer implementation is in progress, though.

We also use these algebras to get new proofs of known results. For instance,
using the classical notations for the parameters of the polynomial,
the classical form of the R-matrix (recalled in section \ref{sectLG}) implies that the Links-Gould polynomial
takes values in $\Z[t_0^{\pm \frac{1}{2}},t_1^{\pm \frac{1}{2}}, \sqrt{(t_0-1)(1-t_1)}]$.
Using our results we provide another proof of the following theorem of Ishii.

\begin{theor} (Ishii, \cite{ISHIIALEX} theorem 1) The Links-Gould polynomial takes values in $\Z[t_0^{\pm 1}, t_1^{\pm 1}]$.
\end{theor}
\begin{proof} The algebras $A_3$ and $A_4$ are defined and are split semisimple over $\Q(t_0,t_1)$. Moreover, the decomposition of $A_{n+1}$ as a $A_n$-module described in section \ref{Markovtrace}
 is valid over the field $\Q(t_0,t_1)$, and the
unicity of the Markov trace proves that it is also defined over $\Q(t_0,t_1)$. It follows that the Links-Gould invariant
takes value in $\Q(t_0,t_1)$. The conclusion follows from the elementary fact that
$\Z[t_0^{\pm \frac{1}{2}},t_1^{\pm \frac{1}{2}}, \sqrt{(t_0-1)(1-t_1)}] \cap \Q(t_0,t_1) = \Z[t_0^{\pm 1},t_1^{\pm 1}]$.
\end{proof}

Finally, we investigate the image and kernel of the morphism $B_n \to LG_n^{\times}$,
and get the following (see section \ref{sectimagenoyau}).

\begin{theor} \label{theoimB} The morphism $B_n \to LG_n^{\times}$ is injective. If $LG_n = \bigoplus_{1 \leq k \leq N} Mat_{c(k)}(K)$
is the decomposition of $LG_n$ in a sum of matrix algebras, the Zariski closure of $B_n$
inside $LG_n^{\times}$ contains $\bigoplus_{1 \leq k \leq N} SL_{c(k)}(K)$.
\end{theor}

Finally, the question of whether these conjecturally equal algebras are \emph{the} `right' algebras, in the sense
of being the minimal ones, for the
Links-Gould polynomial, is equivalent to the following one, that we leave open for the
time being. 
\begin{question} Is this Markov trace non-degenerate on $A_n$ ? on $LG_n$ ?
\end{question}

We prove that it is non-degenerate on $A_4 = LG_4$, as it is a linear combination
with non-zero coefficients of matrix traces (see table \ref{tablecoefs}), thus providing some
evidence to a `yes' answer to this question. A complete answer of this question for $LG_n$ should be given
by the complete determination of the similar coefficients for arbitrary $n$, which is an interesting task in itself.

\begin{figure}
\begin{center}
\resizebox{!}{7cm}{\includegraphics{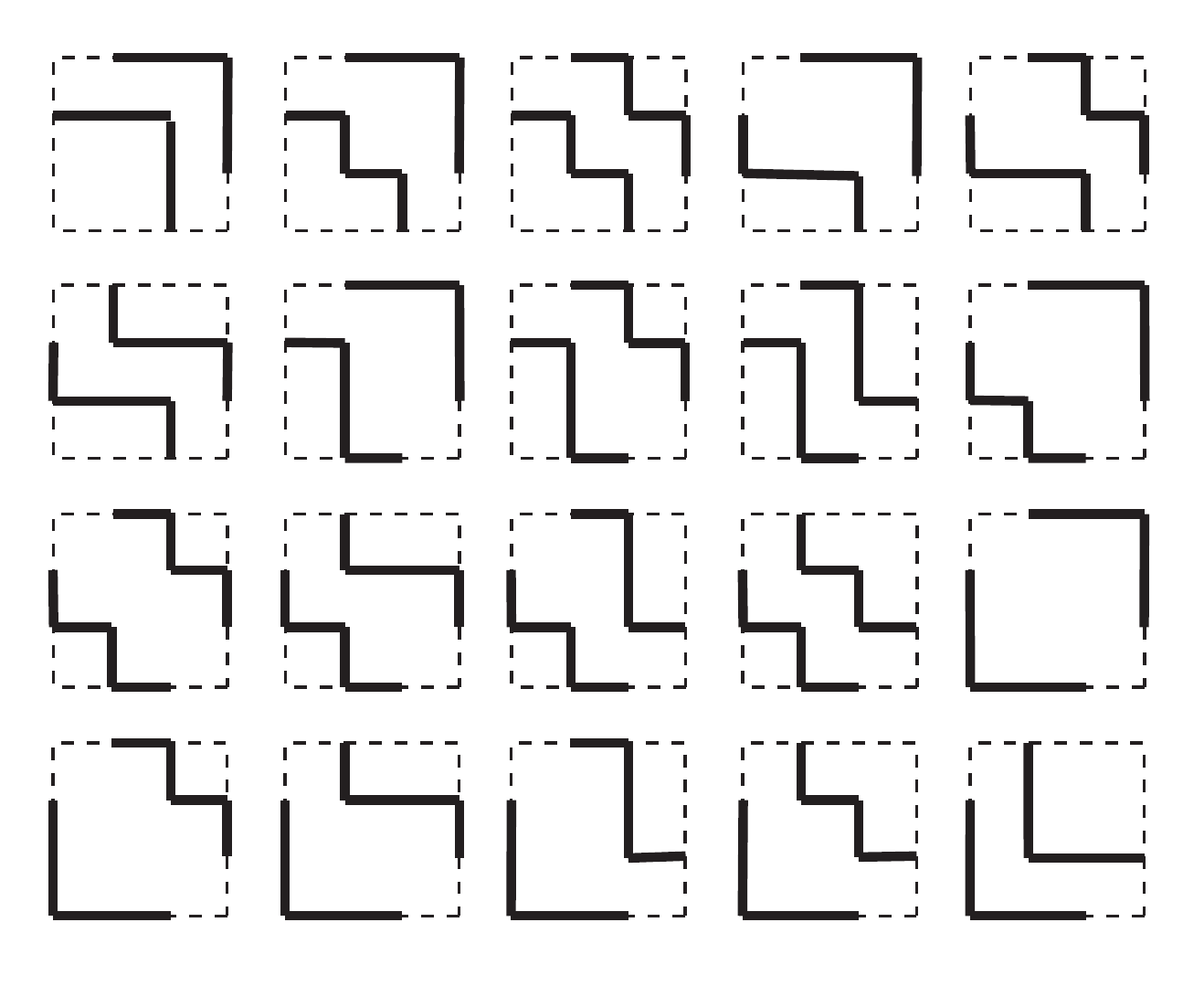}}
\end{center}
\caption{Non-crossing paths as a potential combinatorial model for $LG_3$.}
\label{figchemsLG3}
\end{figure}

{\bf Acknowledgements.} I.M. thanks C. Gruson and P. Vogel for useful discussions on Lie superalgebras. E.W. thanks N. Geer for fruitful email correspondence on the
Links-Gould invariant. E.W. has been partially supported by a FABER
Grant X110CVHCP-2011 and by the French ANR project
ANR-11-JS01-002-01. I.M. has benefited from the ANR grant ANR-09-JCJC-0102-01, corresponding to the ANR project `RepRed'.

\section{Quantum definitions and properties}


\subsection{Lie superalgebras}

We recall the basic notions about Lie superalgebras. From
now on the vector spaces under consideration are defined over
a field $\kk$ of characteristic $0$. For simplicity we moreover
assume $\kk \subset \C$. A \emph{superspace} $V$
is a $\Z/2\Z$-graded space $V = V_0 \oplus V_1$. 
An homogeneous
element in $V$ is called even if it belongs to $V_0$, odd
if it belongs to $V_1$. We denote $|a| \in \{ 0, 1 \}$ the degree
of an homogeneous element.
 A \emph{superalgebra}
$A$ is an associative unital $\Z/2\Z$-graded algebra (that is
$A = A_0 \oplus A_1$ with $A_i A_j \subset A_{i+j}$).
The tensor superproduct $A \otimes_s B$ of two superalgebras $A,B$
is defined as the vector space $A \otimes B$ endowed with
the natural $(\Z/2\Z)$-graduation and by the
multiplication $(a \otimes b)(c \otimes d) = (-1)^{|b||c|}(ac \otimes bd)$
for homogeneous $b,c$. It is straightforward to check
that $((A \otimes_s B) \otimes_s C)$ is naturally isomorphic
to $(A \otimes_s(B \otimes_s C))$, hence the superproduct
$A^{(1)} \otimes_s A^{(2)} \otimes_s \dots \otimes_s A^{(n)}$ of an ordered
(finite) family of superalgebras $A^{(1)},\dots,A^{(n)}$ is well-defined.

For a superspace $V = V_0 \oplus V_1$,
the algebra $\End(V)$ has a natural superalgebra structure, with
$\End(V)_0 = \End(V_0) \oplus \End(V_1)$ and
$\End(V)_1 = \Hom(V_0,V_1) \oplus \Hom(V_1,V_0)$. For
$V, W$ two superspaces, the tensor product $V \otimes W$ has
a natural superspace structure $(V \otimes W)_0 = (V_0 \otimes W_0) \oplus (V_1
\otimes W_1)$, $(V \otimes W)_1 = (V_0 \otimes W_1)\oplus ( V_1
\otimes W_0)$. Assuming that $V$ and $W$ are finite dimensional, 
there is an isomorphism $\End(V) \otimes_s \End(W) \simeq
\End(V \otimes W)$. It associates to $a \otimes b \in \End(V) \otimes_s \End(W)$,
with homogeneous $a,b$, the endomorphism of $V \otimes W$
that maps $v \otimes w$ to $(-1)^{|b||v|} (av) \otimes (bw)$
for homogeneous $v,w$.
It can be extended to an isomorphism of superalgras $\End(V)^{\otimes_s n} = 
\End(V)\otimes_s\End(V) \otimes_s \dots \otimes_s \End(V)$
with $\End(V \otimes \dots \otimes V) = \End(V^{\otimes n})$.

For the axiomatic definition of Lie superalgebras we refer to \cite{KACSUPER},
and recall the lazy definition as a graded subspace of
some associative superalgebra stable under the
superbracket defined by $[a,b] = ab - (-1)^{|a||b|} ba$ for
homogeneous elements $a,b$. To each Lie superalgebra
$\g$ is associated its universal envelopping algebra $\U \g$. It is
a superalgebra with a (super)coproduct $\Delta : \U \g \to
\U \g \otimes_s \U \g$ defined by $\Delta(a) = a \otimes 1 +
1 \otimes a$ for $a \in \g$ (the definition of the `diagonal homomorphism'
of \cite{KACSUPER} is well-known to be flawed). We let
$\Delta_n : \U \g \to (\U \g)^{\otimes_s n}$ denote
the iterated coproduct.


\subsection{Casimir operators}
\label{ssdefcasimir}

Let $\g$ be a finite-dimensional Lie superalgebra, with
a non-degenerate bilinear form $<\ , \ >$ which is
invariant, that is $<[a,b],c> = <a,[b,c]>$), supersymmetric,
that is $<b,a> = (-1)^{|a||b|} <a,b>$ for homogeneous $a,b$,
and consistent, that is $<\g_0,\g_1> = 0$. Let $e_1,\dots,e_n$
be an homogeneous basis of $\g$, and denote $|i| = |e_i|$.
The dual basis $(\tilde{e}_i)$ is defined by
$< \tilde{e}_i,e_j> = \delta_{ij}$ (Kronecker symbol),
and $|\tilde{e}_i| = |e_i| = |i|$ by consistency. By supersymmetry,
$\stackrel{\approx}{e}_i = (-1)^{|i|} e_i$. Let $\Omega = \sum_i
e_i \otimes \tilde{e}_i  \in \g \otimes \g$. It is independent
of the choice of the basis, as it is the image of $<\ , \ >$
under the isomorphism of vector spaces
$(\g \otimes \g)^* \simeq \g^* \otimes \g^* \simeq  \g \otimes \g$
induced by $<\ , \ >$. In particular,
$\Omega = \sum_i \tilde{e}_i \otimes \stackrel{\approx}{e}_i 
= \sum_i (-1)^{|i|} \tilde{e}_i \otimes e_i$.

The Casimir operator
is the element $C = \sum_i e_i \tilde{e}_i = \sum_i (-1)^{|i|} \tilde{e}_i e_i \in \U \g$. It does not
depend on the choice of the basis either. Note that
$C \in \U \g$ is even.

Let $n \geq 2$. For $i<j \leq n$, we let
$$
\Omega_{ij} = \sum_r 1 \otimes \dots \otimes e_r \otimes 1 \otimes \dots \otimes 1 \otimes \tilde{e}_r \otimes 1
\otimes \dots \otimes 1 \in \U \g ^{\otimes_s n}
$$ 
where the $e_r$'s are in position $i$ and the $\tilde{e}_r$'s are
in position $j$. The following lemma is standard in
the classical (that is, Lie algebra) case. Its extension to the
`super' case is straightforward.

\begin{lemma} \label{lemCT}
$$
2(\Omega_{1,n} + \Omega_{2,n} + \dots + \Omega_{n-1,n}) = \Delta_n(C) - \Delta_{n-1}(C) \otimes 1 - 1 \otimes \dots \otimes 1 \otimes C
$$
\end{lemma}
\begin{proof} It is easily checked that $\sum_{i<n} \Omega_{i,n}=
\sum_r \Delta_{n-1}(e_r) \otimes \tilde{e}_r$. On the other
hand, $\Delta_n(C) = \sum_r \Delta_n(e_r) \Delta_n(\tilde{e}_r)$ is
$$
\sum_r \left( \sum_i 1 \otimes \dots \otimes \underbrace{e_r}_{(i)} \otimes \dots \otimes 1 \right)
\left( \sum_j 1 \otimes \dots \otimes \underbrace{\tilde{e}_r}_{(j)} \otimes \dots \otimes 1 \right)
$$
that is the sum over $r$ of
$$
\sum_{i<j} 1 \otimes \dots\otimes  \underbrace{e_r}_{(i)}
\otimes \dots \otimes \underbrace{\tilde{e}_r}_{(j)} \otimes \dots \otimes 1
+ \sum_{j<i} (-1)^{|e_r||\tilde{e}_r|}1 \otimes \dots\otimes  \underbrace{\tilde{e}_r}_{(j)}
\otimes \dots \otimes \underbrace{e_r}_{(i)} \otimes \dots \otimes 1
$$
and of
$$
\sum_i 1 \otimes \dots \otimes \underbrace{e_r \tilde{e}_r}_{(i)} \otimes \dots \otimes 1 
$$
It follows that
$$
\Delta_n(C) - \Delta_{n-1}(C) \otimes 1 = 
1^{\otimes(n-1)} \otimes C + \sum_r \Delta_{n-1}(e_r) \otimes \tilde{e}_{r}
+ \sum_r \Delta_{n-1}(\tilde{e}_r) \otimes \stackrel{\approx}{e}_{r}
$$
since $|e_r||\tilde{e}_r| = |r|^2 = |r|$ and $\stackrel{\approx}{e}_{r} =
(-1)^{|r|} e_r$, this proves the lemma.
\end{proof}

We recall that the $\Omega_{ij}$ satisfy the infinitesimal braids or 4T-relations,
namely the 1-form
$$
\sum_{i<j} \Omega_{i,j} \mathrm{d log}(z_i - z_j)
$$
is integrable on $\C_*^n = \{ (z_1,\dots,z_n) \in \C^n \ | \ z_i \neq  z_j \}$.
There is a Kohno-Drinfeld type theorem due to N. Geer (see \cite{GEER})
saying that the representations of the braid group obtained in this
way are isomorphic to the ones originating from the Yamane quantization
of the corresponding Lie superalgebra.

\subsection{The Lie superalgebras $\sl(2|1)$}
When $V$
is a superspace with $V_0 = \kk^m$ and $V_1 = \kk^n$ the
Lie superalgebra associated to $\End(V)$ is traditionnaly denoted
$\gl(m|n)$. It admits a linear form called the
supertrace $str : \gl(m|n) \to \kk$ defined by $str(\End(V)_1) = 0$
and $str_{| \End(V)_0} = \Id_{\End(V_0)} - \Id_{\End(V_1)}$.
Choosing the natural homogeneous basis of $V = \kk^m \oplus \kk^n$
the elements of $\End(V)$ are represented by matrices $\left(\begin{array}{cc}
\alpha & \beta \\ \gamma & \delta
\end{array} \right)$, with supertrace $\tr \alpha - \tr \delta$.
The Lie superalgebra structure of $\gl(m|n)$ restricts to a Lie
superalgebra structure on $\Ker str$, denoted $\sl(m|n)$. We assume
$m \neq n$. Then $<a,b> = str(ab)$ defines on $\sl(m|n)$
a bilinear form fulfilling the requirements of section \ref{ssdefcasimir}.

We now specialize to the case $m = 2$, $n=1$, and choose for basis
of $\sl(2|1)$ the following elements
$$
h_1 = \left( \begin{array}{cc|c} 1 & 0 & 0 \\ 0 & -1 & 0 \\ \hline 0 & 0 & 0\end{array} \right) 
h_2 = \left( \begin{array}{cc|c} 0 & 0 & 0 \\ 0 & 1 & 0 \\ \hline 0 & 0 & 1\end{array} \right) 
e_1 = \left( \begin{array}{cc|c} 0 & 1 & 0 \\ 0 & 0 & 0 \\ \hline 0 & 0 & 0\end{array} \right) 
f_1 = \left( \begin{array}{cc|c} 0 & 0 & 0 \\ 1 & 0 & 0 \\ \hline 0 & 0 & 0\end{array} \right) 
$$
{\ }
$$
e_2 = \left( \begin{array}{cc|c} 0 & 0 & 0 \\ 0 & 0 & 1 \\ \hline 0 & 0 & 0\end{array} \right) 
f_2 = \left( \begin{array}{cc|c} 0 & 0 & 0 \\ 0 & 0 & 0 \\ \hline 0 & 1 & 0\end{array} \right) 
[e_1,e_2] = \left( \begin{array}{cc|c} 0 & 0 & 1 \\ 0 & 0 & 0 \\ \hline 0 & 0 & 0\end{array} \right) 
[f_1,f_2] = \left( \begin{array}{cc|c} 0 & 0 & 0 \\ 0 & 0 & 0 \\ \hline -1 & 0 & 0\end{array} \right) 
$$
The elements $h_1,h_2,f_1,f_2$ are even, the other four are odd.
We have $[h_i,h_j] = 0$, $[h_i,e_j] = a_{ij} e_j$, $[h_i,f_j] = -a_{ij} f_j$,
$[e_i,f_j] = \delta_{ij} h_i$, $[e_2,e_2]=[f_2,f_2] = 0$,
$[e_1,[e_1,e_2]] = [f_1,[f_1,f_2]] = 0$ with $(a_{ij})_{i,j}$ the Cartan
matrix $\left( \begin{array}{cc} 2 & -1 \\ -1 & 0 \end{array} \right)$,
so our notations are compatible with \cite{PATGEER}.

These 8 elements are weight vectors under the adjoint representation
for the Cartan subalgebra $\h$ spanned by $h_1,h_2$. Clearly $\h$
is a subspace of the space $\mathfrak{d}$ of the diagonal matrices
in $\gl(2|1)$. Following \cite{KACREPS} we let $\eps_1^K,\eps_2^K, \delta_1^K \in \mathfrak{d}^*$
denote the dual basis of the diagonal vectors $E_{11},E_{22},E_{33}$.
Under the natural restriction map $\mathfrak{d}^* \to \h^*$ one has
$\eps_1^K \mapsto h_1^*$, $\eps_2^K \mapsto h_2^* - h_1^*$, $\delta_1^K \mapsto h_2^*$,
where $(h_1^*,h_2^*)$ denotes the dual basis of $(h_1,h_2)$.
Recall that a root $\alpha \in \h^*$ is called even (odd) if the corresponding eigenspace
is even (odd). The only positive even root is thus $\alpha_1 = \eps_1^K - \eps_2^K = 2 h_1^* - h_2^*$,
the odd ones are  $\alpha_2 = \eps_2^K - \delta_1^K = - h_1^*$,
$\alpha_1 + \alpha_2 = \eps_1^K- \delta_1^K = h_1^* - h_2^*$. It follows
that the `super half-sum' of the positive roots is $\rho = \frac{1}{2} ( \alpha_1 - (\alpha_2 + \alpha_1 + \alpha_2)
= - \alpha_2 = h_1^*$.

\subsection{Highest weight modules for $\sl(2 | 1)$}

To each $\la \in \h^*$ one can associate a Kac module $V(\la)$ for $\sl(2|1)$, which
is an irreducible heighest weight module, if $< \la + \rho, \alpha > \neq 0$ for odd root $\alpha$ (see \cite{KACREPS}). Such
a weight is called \emph{typical}.
The Cartan matrix is the matrix in the basis $(h_1,h_2)$ of the Killing form $a,b \mapsto <a,b> = str(ab)$
to the Cartan subalgebra $\h$. We have $h_1^* = < -h_2,\cdot >$ and $h_2^* = <-h_1 - 2h_2, \cdot >$.
It follows that the induced bilinear form $<\ , \ >$ on $\h^*$ satisfies
$<h_1^*,h_1^*> = 0$, $<h_1^*,h_2^*> = -1$, $<h_2^*,h_2^*> = -2$. In particular,
for $\la = a_1 h_1^* + a_2 h_2^*$ one computes $< \la + \rho, -h_1^*> = a_2$
and $< \la + \rho, h_1^* - h_2^* > = a_1 + a_2 + 1$. Thus $\la$ is typical if $a_2 \neq 0$
and $a_1 + a_2 + 1 \neq 0$.

In that case, the (usual) dimension of $V(\la)$ is $4 < \la  - \alpha_2,\alpha_1> = 4(a_1+1)$,
and its odd and even parts have the same dimension $2(a_1+1)$.

Finally, the value of $C$ on $V(\la)$ for a typical $\la$ is
$<\la,\la + 2 \rho>$. We find $C_{|V(\la)} = -2 a_2(a_1+a_2+1)$.

Letting $V(a_1,a_2)$ denote the highest weight module with weight $a_1 h_1^* + a_2 h_2^*$, which is typical
for $a_2 \neq 0$ and $a_1 + a_2 + 1 \neq 0$. If $a_1 \neq 0$, then $V(a_1,a_2) \otimes V(0,b_2)$ is isomorphic
to
$$
V(a_1 ,a_2+b_2) \oplus V(a_1+1,a_2+b_2) \oplus V(a_1-1,a_2+b_2+1) \oplus V(a_1,a_2+b_2+1)
$$
whenever all the weights involved are typical. When $a_1 = 0$, and under the same conditions, it
is isomorphic to $V(a_1 ,a_2+b_2) \oplus V(a_1+1,a_2+b_2) \oplus  V(a_1,a_2+b_2+1)$ (see \cite{PATGEER} lemma 1.3  or \cite{GRUSON} proposition 4.1).

\subsection{The bimodules $V(0,\alpha)^{\otimes 2}$} \label{sectVV}

According to the Clebsh-Gordan decomposition above,
for generic values of $\alpha$ we have
$V(0,\alpha)\otimes V(0,\alpha) = V(0,2 \alpha) \oplus V(0,2 \alpha+1) \oplus
V(1,2 \alpha)$,
with $C$ taking values $-4 \alpha(2 \alpha+1)$, $-4(\alpha+1)(2 \alpha+1)$,
$-8 \alpha(\alpha+1)$ and the supersymmetrizer $\tau$ the values $1,1,-1$
(see \cite{LINKSGOULD}). Since, on $V(0, \alpha)$,
$C$ takes value $-2\alpha(\alpha+1)$, then $\Omega_{12}$
takes the values $-2 \alpha^2$, $-2(\alpha+1)^2$, $-2\alpha (\alpha+1)$.

\section{The Links-Gould invariant of links}
\label{sectLG}
In this section, we give a definition of the Links-Gould invariant \cite{LINKSGOULD} of links. It is an 2-variable invariant constructed by using the Reshetikhin-Turaev recipe \cite{ReshTur} applied to a one-parameter family of representations of the quantum supergroup $\mathsf{U}_q\sl(2,1)$ -- or $\mathsf{U}_q \gl(2|1)$, which
does not make any difference, see e.g. the introduction of \cite{ISHIIDEWIT}. It is naturally defined to be an invariant of (1-1)-tangle \cite{DeWitKauffLinks} and turn's out to be in fact an invariant of links (see Remark \cite{PATGEER}). For convenience we define following closely  Ishii \cite{Ishiisplit} the Links-Gould invariant as a partial trace on the braid group.\\

Recall that it follows from Alexander theorem \cite{Alexander} that any link can be presented as the closure of a braid. In addition Markov's theorem \cite{Markov} precises exactly when two braids represent the same link. Hence one can always define a link invariant on a braid closure representative, as long as it is invariant under the Markov moves (see also Section \ref{Markovtrace}).\\

Let $K$ be a field of characteristic $0$ containing two algebraically
independent elements $t_0,t_1 \in K^{\times}$, as well as the
square roots $\sqrt{t_0}$, $\sqrt{t_1}$ and $\sqrt{ (t_0-1)(1-t_1)}$.
Consider a four-dimensional vector space $V$ over $K$ with basis $(e_1,e_2,e_3,e_4)$. Recall that a {\it R-matrix}  $R$ is an invertible element of  $\rm{End}(V\otimes V)$ that satisfies the Yang-Baxter equation:
$$(\mbox{Id}_V\otimes R)(R\otimes \mbox{Id}_V)(\mbox{Id}_V\otimes R)=(R\otimes \mbox{Id}_V)(\mbox{Id}_V\otimes R)(R\otimes \mbox{Id}_V).$$
An endomorphism $A$ of $V\otimes V$ will be represented by a matrix $(A_{i,j})_{i,j\in[|1,16|]}$ where 
$$A_{4(i-1)+j,4(k-1)+l}=A_{k,l}^{i,j} \mbox{ with } A(e_k\otimes e_l)=\sum_{i,j=1}^4 A_{k,l}^{i,j}e_i\otimes e_j\mbox{ for all }k,l=1,\ldots,4.$$
The endomorphism $R\in\mbox{End}(V\otimes V)$ given by the following matrix $(R_{i,j})_{i,j\in[|1,16|]}$ \cite{DeWitKauffLinks}  is a R-matrix :
\begin{center}
\scalebox{0.8}{$\left(\begin{array}{cccc|cccc|cccc|cccc}
t_0 & \cdot & \cdot & \cdot & \cdot & \cdot & \cdot & \cdot & \cdot & \cdot & \cdot & \cdot & \cdot & \cdot & \cdot & \cdot \\
\cdot & \cdot & \cdot & \cdot  & t_0^{1/2} & \cdot & \cdot & \cdot & \cdot & \cdot & \cdot & \cdot & \cdot & \cdot & \cdot \\
 \cdot & \cdot & \cdot & \cdot & \cdot & \cdot & \cdot & \cdot & t_0^{1/2} & \cdot & \cdot & \cdot & \cdot & \cdot & \cdot & \cdot \\
 \cdot & \cdot & \cdot & \cdot & \cdot & \cdot & \cdot & \cdot & \cdot & \cdot & \cdot & \cdot & 1 & \cdot & \cdot & \cdot \\
 \hline
  \cdot & t_0^{1/2} &  \cdot & \cdot & t_0-1 &  \cdot & \cdot & \cdot & \cdot & \cdot & \cdot & \cdot & \cdot & \cdot & \cdot & \cdot \\
   \cdot & \cdot & \cdot & \cdot & \cdot &-1 & \cdot & \cdot & \cdot & \cdot & \cdot & \cdot & \cdot & \cdot & \cdot & \cdot \\
    \cdot & \cdot & \cdot & \cdot & \cdot & \cdot &t_0t_1-1 & \cdot & \cdot &-t_0^{1/2}t_1^{1/2} & \cdot & \cdot & -t_0^{1/2}t_1^{1/2}Y & \cdot &\cdot & \cdot \\
 \cdot & \cdot & \cdot & \cdot &  \cdot &  \cdot &  \cdot &  \cdot &  \cdot &  \cdot &  \cdot &  \cdot &  \cdot & t_1^{1/2} & \cdot &  \cdot\\
 \hline   
  \cdot & \cdot & t_0^{1/2}  & \cdot &  \cdot &  \cdot & \cdot & \cdot & t_0-1 & \cdot &  \cdot &  \cdot &  \cdot &  \cdot &  \cdot &  \cdot \\
   \cdot &  \cdot & \cdot &  \cdot &  \cdot &  \cdot & -t_0^{1/2}t_1^{1/2} & \cdot &  \cdot & \cdot & \cdot & \cdot & Y &  \cdot & \cdot & \cdot\\
    \cdot & \cdot & \cdot & \cdot & \cdot & \cdot & \cdot & \cdot & \cdot & \cdot & -1 &  \cdot &  \cdot & \cdot & \cdot & \cdot \\
     \cdot & \cdot & \cdot & \cdot & \cdot & \cdot & \cdot & \cdot & \cdot & \cdot & \cdot & \cdot & \cdot & \cdot & t_1^{1/2}  \cdot\\
     \hline
      \cdot & \cdot & \cdot & 1&  \cdot & \cdot & -t_0^{1/2}t_1^{1/2}Y &  \cdot & \cdot & Y &  \cdot & \cdot & Y^2  & \cdot &  \cdot & \cdot\\
       \cdot & \cdot & \cdot & \cdot & \cdot & \cdot & \cdot & t_1^{1/2} &  \cdot & \cdot & \cdot & \cdot & \cdot & t_1-1 &  \cdot & \cdot \\
        \cdot & \cdot & \cdot & \cdot & \cdot & \cdot & \cdot & \cdot & \cdot & \cdot & \cdot & t_1^{1/2} & \cdot & \cdot &t_1-1  \cdot \\
         \cdot & \cdot & \cdot & \cdot & \cdot & \cdot & \cdot & \cdot & \cdot & \cdot & \cdot & \cdot & \cdot & \cdot & \cdot & t_1 
 
   \end{array}\right),$}
   \end{center}

where $Y=((t_0-1)(1-t_1))^{\frac{1}{2}}$.\\

Let $B_n$ be the braid group on $n$ strands with Artin generators $s_1,\ldots,s_{n-1}$. A R-matrix $R\in\mbox{End}(V\otimes V)$ defines a representation $\rho_R^n$ of the braid group in $\mbox{Aut}(V^{\otimes n})$ by sending $s_i$ to $\mbox{Id}_V^{\otimes (i-1)}\otimes R \otimes \mbox{Id}_V^{\otimes (n-1-i)}\in \mbox{Aut}(V^{\otimes n})$, for $i=1,\ldots,n-1$. In addition denote $\mu$ the automorphism of $V$  defined by $\mu(e_i)=\mu_i e_i$ for $i=1,\ldots 4$ where $(\mu_1,\mu_2,\mu_3,\mu_4)=(t_0^{-1}, -t_1, -t_0^{-1}, t_1)$.\\

Given $A\in \mbox{End}(V^{\otimes n})$ such that 
$$A(e_{i_1}\otimes \cdots \otimes e_{i_n})=\sum_{1\leq j_1,\ldots, j_n\leq 4}A_{i_1,\ldots,i_n}^{j_1,\ldots,j_n}e_{j_1}\otimes\cdots \otimes e_{j_n},$$
define for $m=1,\ldots,n$ the partial trace operator $tr_n^m(A)\in\mbox{End}(V^{\otimes (n-m)})$ by
$$tr_n^m(A)(e_{i_1}\otimes \cdots \otimes e_{i_{n-m}})=\sum_{1\leq j_1,\ldots,j_n\leq 4}A_{i_1,\ldots,i_{n-m},j_{n-m+1}\ldots,j_n}^{j_1,\ldots,j_{n-m},j_{n-m+1},\ldots,j_n}e_{j_1}\otimes\cdots\otimes e_{j_{n-m}}.$$ 
Given a braid $\beta\in B_n$ we denote by $\widehat{\beta}$ the closure of $\beta$ (see Figure (\ref{cl})) we define the Links-Gould invariant $\mathfrak{LG}(\widehat{\beta};t_0,t_1)$ of the link $\widehat{\beta}$ by the following formula:
$$\mbox{tr}_n^{n-1}((\mbox{Id}_V\otimes \mu^{\otimes (n-1)})(\rho_R^n(\beta))=\mathfrak{LG}(\widehat{\beta};t_0,t_1)\mbox{Id}_V.$$

\begin{center}
\begin{figure}[h!]
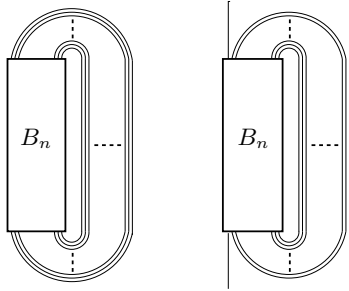
\caption{Closure of a braid and partial closure of a braid.}
\label{cl}
\end{figure}
\end{center}

The Links-Gould invariant can also be defined by considering the right picture of Figure (\ref{cl}) which describes a topological partial trace and use the tangle invariant construction described in \cite{DeWitKauffLinks} (notice that one just needs to assign an endomorphism to each cup or cap). These two constructions coincide up to a change of variables.

The R-matrix that we introduced here has for eigenvalues $-1,t_0,t_1$.
Rescaling the R-matrix construction by $s_i \mapsto (-a) s_i$ with $a \in K^{\times}$,
we get a representation $B_n \to \GL(V^{\otimes n})$ in which the image of $s_i$ is annihilated by the cubic polynomial $(X-a)(X-b)(X-c)$, with $b = -t_0 a$, $c = -t_1 a$.
We use this renormalization from now on.


\section{The morphism $K B_n \onto LG_n(\alpha)$}
\label{sectstruct}



We take $\g = \sl(2|1)$ and we let $LG_n(\alpha)$ denote the commutant algebra of the image of $\U \g$ inside
$\End(V(0,\alpha)^{\otimes n})$.  Note that $LG_n(\alpha) \subset LG_{n+1}(\alpha)$
under $m \mapsto m \otimes 1$.

It contains the group algebra $\kk \mathfrak{S}_n$ of the symmetric group, as well
as the image $\mathcal{LG}_n(\alpha)$ of the Lie algebra $\mathcal{T}_n$
of infinitesimal braids (or horizontal chord diagrams), also known as the
holonomy Lie algebra of the space $\C_*^n$. Recall that this algebra $\mathcal{T}_n$
is generated by elements $t_{ij}$ which are dual to the 1-forms $\mathrm{dlog}(z_i-z_j)$,
and which satisfy the relations $t_{ij} = t_{ji}$, $[t_{ij},t_{ik}+t_{kj}] = 0$ and $[t_{ij},t_{kl}] = 0$
whenever $\# \{i,j,k,l \} = 4$. The morphism $\mathcal{T}_n \to \mathcal{LG}_n(\alpha)$
is given by $t_{ij} \mapsto \Omega_{ij}$.

Let $T_n$ be the image of $\sum_{1 \leq i,j\leq n} t_{ij}$ in $\mathcal{LG}_n(\alpha)$.

We let $[a,k]_r = a h_1^* + (r \alpha + k) h_2^*$ for $0 \leq a \leq r-1$, $0 \leq k \leq r-a-1$,
that we identify with the corresponding module $V(a,k+r \alpha)$. 
We also replace $\oplus$ by $+$ when it lightens notations.
Notice that $r \alpha + k \neq 0$ and $a+ r \alpha + k +1 \neq 0$ for $\alpha \not\in  \Q_{<0}$.
The tensor product decomposition above is thus translated as $[a,k]_r \otimes V(0,\alpha) = 
[a-1,k+1]_{r+1} + [a,k]_{r+1} + [a,k+1]_{r+1} + [a+1,k]_{r+1}$, under the convention
that $[a,l]_{r+1}$ is omitted if $a < 0$.
From this we get that, if $a \geq 2$,

$$
\begin{array}{lcl}
[a,k]_r \otimes V(0,\alpha) ^{\otimes 2} &=& 
[a-2,k+2]_{r+2} + 2[a-1,k+1]_{r+2} + 2[a-1,k+2]_{r+2} \\
& &  + 4[a,k+1]_{r+2}+
 [a,k]_{r+2} + 2 [a+1,k]_{r+2} \\
& & + [a,k+2]_{r+2} +2 [a+1,k+1]_{r+2}+
  [a+2,k]_{r+2}  
\end{array}
$$
and in particular $[0,0]_1 \otimes V(0,\alpha) = [0,0]_2 \oplus [0,1]_2 \oplus [1,0]_2$.

We note the following decompositions
$$
\begin{array}{lclr}
{[}a,k]_r \otimes V(0,2 \alpha) &=& [a,k]_{r+2} + [a+1,k]_{r+2} + [a-1,k+1]_{r+2}+  [a,k+1]_{r+2}& \mbox{ if } a \neq 0 \\
{[}0,k]_r \otimes V(0,2 \alpha) &=& [0,k]_{r+2} + [1,k]_{r+2} +  [0,k+1]_{r+2}& \\
{[}a,k]_r \otimes V(0,2 \alpha+1) &=& [a,k+1]_{r+2} + [a+1,k+1]_{r+2} + [a-1,k+2]_{r+2}+  [a,k+2]_{r+2}& \mbox{ if } a \neq 0 \\
{[}0,k]_r \otimes V(0,2 \alpha+1) &=& [0,k+1]_{r+2} + [1,k+1]_{r+2} +  [0,k+2]_{r+2}& \\
\end{array}
$$
Since $V(0,\alpha) \otimes V(0,\alpha) = 
V(0,2 \alpha) + V(0,2 \alpha+1) + V(1,2\alpha)$ it
follows that
$$
\begin{array}{lcl}
[a,k]_r \otimes V(1,2 \alpha) &=& 
[a-2,k+2]_{r+2} + [a-1,k+1]_{r+2} + [a-1,k+2]_{r+2} 
  + 2[a,k+1]_{r+2} \\ & & +
  [a+1,k]_{r+2} 
 + [a+1,k+1]_{r+2}+
  [a+2,k]_{r+2}  
\end{array}
$$

\begin{theor} \label{theosurjcommut} For generic values of $\alpha$, the image of $\U \mathcal{T}_n$
in $\End(V(0,\alpha)^{\otimes n})$ is $LG_n(\alpha)$.
\end{theor}
\begin{proof}
The techniques used to prove the result are inherited from \cite{THESE}, to which the reader can report for
a detailed description of the general setting.
The value of $C$ on $[a,k]_r$ is $-2(r \alpha + k) (r\alpha + k + a + 1)$.
It is a polynomial of degree 2 in $\alpha$ with dominant term $-2 r^2$, hence it is uniquely
defined by its roots $-k/r$ and $-(k+a+1)/r$. For a given $r$, no two such
polynomial can be equal, as $\{ -k/r, -(k+a+1)/r \} = \{ -k'/r, -(k'+a'+1)/r \}$
iff $\{ k, k+a+1 \} = \{ k', k'+a'+1 \}$, and $k$ is determined by $k = \min \{ k, k+a+1 \}$.

We now consider a given $r$, and the $\U \mathcal{T}_r$ module
$\Hom_{\U \g} ([a,k]_r,V(0,\alpha)^{\otimes n})$, with $\g = \sl(2|1)$. We choose
a basis compatible with the Bratteli diagram (that is, a basis of highest weight
vectors for all $\Delta_s(\g) \otimes 1^{\otimes (r-s)}$ for all $s \leq r$). We call
such a basis a standard basis. See table \ref{brattelir5} for the beginning of the
Bratteli diagram.

The above remark on the Casimir together with lemma \ref{lemCT} imply that,
when restricted to the commutative Lie subalgebra $\mathcal{D}_r$ of $\mathcal{T}_r$ generated
by the $t_{1,r} + \dots + t_{r-1,r}$, the restriction of this $\mathcal{T}_r$-module
is the multiplicity-free sum of 1-dimensional representations, this decomposition
being necessarily given by a standard basis. 
In particular, the image of $\U \mathcal{D}_r$ in such a basis is the whole algebra of diagonal matrices.
In order to prove the theorem, we moreover only need to show that such a $\U \mathcal{T}_r$-module
is irreducible, as the eigenvalues of $\sum \Omega_{ij}$ determine the value
of $\Delta_r(C)$ hence the module $[a,k]_r$.

In order to prove the irreducibility, we make the following remark. Let
$t_r = t_{1,r} + \dots + t_{r-1,r}$ and $Y = t_r - t_{r-1}$,
$s$ the transposition $(r-1, r)$, and $u = t_{r-1,r}$. We introduce
the subalgebra $L$ of $\mathfrak{S}_r \ltimes \U \mathcal{T}_r$ generated by
$s,u$ and $Y$, and $\mathcal{L}$ the Lie subalgebra of $\mathcal{T}_r$
generated by $u$ and $Y$. Note that they commute with $\U \mathcal{T}_{r-2}$, and
that the image of $L$ in $\End_{\U \g}(V(0,\alpha)^{\otimes r})$
is the same as the image of $\U \mathcal{L} \subset \U \mathcal{T}_r$, as the image of $s$ is a polynomial (depending on $\alpha$) of
the image of $u$,
at least for generic values of $\alpha$ (this last assertion needs only
to be checked for $r=2$, and in this case it follows from the spectral decomposition
of $V(0,\alpha)^{\otimes 2}$).

Now, the restriction to $\U \mathcal{T}_{r-2} \times \U \mathcal{L}$ of 
$\Hom_{\U \g} ([a,k]_r,V(0,\alpha)^{\otimes r})$ is
$$
\bigoplus_{b,l} \Hom_{\U \g} ([b,l]_{r-2}, V(0,\alpha)^{\otimes (r-2)}) \otimes \Hom_{\U \g} ([a,k]_r,[b,l]_{r-2} \otimes V(0,\alpha)^{\otimes 2})
$$
We first prove that is sufficient to show that the $\U \mathcal{L}$-modules $\Hom_{\U \g} ([a,k]_r,[b,l]_{r-2} \otimes V(0,\alpha)^{\otimes 2})$ are irreducible. Indeed, any stable subspace of
$\Hom_{\U \g} ([a,k]_r,V(0,\alpha)^{\otimes r})$ need to have for basis a subset $\mathcal{B}_0$ of our given (standard) basis $\mathcal{B}$.
Note that this standard basis is naturally indexed by sequences of $\U \g$-modules of the form $([a_1,k_1]_1, [a_2,k_2]_2,\dots,
[a_r = a, k_r = k]_r)$ or, more combinatorially, by paths in the Bratteli diagram. 

By induction on $r$, we can assume that all the $\Hom_{\U \g}([b,l]_{r-1},V(0,\alpha)^{\otimes (r-1)})$
are irreducible under the action of $\U \mathcal{T}_{r-1}$. 
Since the restriction of $\Hom_{\U \g} ([a,k]_r,V(0,\alpha)^{\otimes r})$ to $\U \mathcal{T}_{r-1}$ is
$$
\bigoplus_{b,l} \Hom_{\U \g}([b,l]_{r-1},V(0,\alpha)^{\otimes (r-1)}) \otimes \Hom_{\U \g}([a,k]_r,[b,l]_{r-1} \otimes V(0,\alpha))
$$
(with $\U \mathcal{T}_{r-1}$ acting trivially on the second tensor factor) this implies that, if some $\gamma = (\gamma_1,\dots,\gamma_r = [a,k]_r)$ belongs to
$\mathcal{B}_0$, then all the elements of $\mathcal{B}$ with the same $\gamma_{r-1}$ belong to $\mathcal{B}_0$. From the
Bratteli diagram it is clear that, if $\mathcal{B}_0 \neq \mathcal{B}$, there exist $\delta \in \mathcal{B} \setminus \mathcal{B}_0$
and $\gamma \in \mathcal{B}_0$ which differ only at the place $r-1$, that is $\gamma_i = \delta_i$ for $i \neq r-1$. But if we know
in addition that the $\U \mathcal{L}$-action on $\Hom_{\U \g}( [a,k]_r, \gamma_{r-2} \otimes V(0,\alpha)^{\otimes 2})
= \Hom_{\U \g}( [a,k]_r, \delta_{r-2} \otimes V(0,\alpha)^{\otimes 2})$ is irreducible, there exists an element of $\U \mathcal{L}$
that maps $\gamma$ and $\delta$, which proves the contradiction that we want.

We now prove that every $M = \Hom_{\U \g}([a,k]_r, [b,l]_{r-2} \otimes V(0,\alpha)^{\otimes 2})$ is
an irreducible $\U \mathcal{L}$-module, or equivalently an irreducible $L$-module. There are several
cases to consider. We contend ourselves with the most difficult (4-dimensional) one,
the other ones being similar and easier (and basically already dealt with in \cite{THESE}, annexe).
This is the case where $[b,l]_{r-2} = [a,k-1]_{r-2}$, with $a \geq 1$. For the action of
$u$ and $s$, $M$ can be decomposed as $D_1 \oplus D_2 \oplus U$ with
$\dim D_1 = \dim D_2 = 1$, $\dim U = 2$, and
$$
\begin{array}{lcl}
D_1 &=& \Hom_{\U \g}([a,k]_r, [a,k-1]_{r-2} \otimes V(0,2\alpha)) \\
D_2 &=& \Hom_{\U \g}([a,k]_r, [a,k-1]_{r-2} \otimes V(0,2\alpha+1)) \\
U &=& \Hom_{\U \g}([a,k]_r, [a,k-1]_{r-2} \otimes V(1,2\alpha) )
\end{array}
$$
We choose a basis of $M$ compatible with this decomposition. In such a basis,
we have
$$
u = -2 \left( \begin{array}{ccccc} \alpha^2 & 0 & 0 & 0 \\  
0 & (\alpha+1)^2 & 0 & 0 \\ 
0 & 0 & \alpha(\alpha+1) & 0 \\ 
0 & 0 & 0&\alpha(\alpha+1)  \end{array} \right) 
\ \ 
s = \left( \begin{array}{ccccc} 1 & 0 & 0 & 0 \\  
0 & 1 & 0 & 0 \\ 
0 & 0 & -1 & 0 \\ 
0 & 0 & 0&-1  \end{array} \right) 
$$
Now, we note that $s,u$ and $Y$ are related
by the (easy-to-check) relation $Y + sYs = 2u$. This relation
implies that $Y$ has the form
$$
\left( \begin{array}{cc} D & M \\  N & v \Id \end{array} \right)
$$
with $M,N,D, \Id \in Mat_2(\kk)$, 
$$
D = -2 \left( \begin{array}{cc} \alpha^2 & 0 \\ 0 & (\alpha+1)^2 \\ \end{array} \right)
\ \ \Id = \left( \begin{array}{cc} 1 & 0 \\ 0 & 1  \end{array} \right)
$$
and $v = -2\alpha(\alpha+1)$. We now compute the eigenvalues of $Y$, using lemma \ref{lemCT}. 
Recall the
part of the Bratteli diagram that is involved here.
$$
\xymatrix{
& & [a,k-1]_{r-2} \ar[dll] \ar[dl] \ar[dr] \ar[drr] \\
[a-1,k]_{r-1} \ar[drr]  & [a,k-1]_{r-1}\ar[dr] & & [a,k]_{r-1}\ar[dl] & [a+1,k-1]_{r-1} \ar[dll] \\
& &  [a,k]_r  \\
}
$$
By lemma \ref{lemCT}, the eigenvalues of $Y$ are the ones
of $(\Delta_r(C) + \Delta_{r-2}(C)\otimes 1 \otimes 1 - 2\Delta_{r-1}(C) \otimes 1 )/2$,
that is $-2(\alpha^2 + \alpha -\frac{a}{2}), -2(\alpha^2 + r \alpha + k + \frac{a}{2}),
-2(\alpha^2 + (2-r) \alpha - \frac{a}{2} - k), -2(\alpha^2 + \alpha + \frac{a}{2}+1)$. Since $a \geq 1$
and $r \geq 3$, we note that $Y$ and $u$ share no common eigenvalue for generic $\alpha$. This implies the following
\begin{itemize}
\item $\det M \neq 0$, for otherwise there would exist some $m = (x,y) \in \kk^2$ in $\Ker M$, hence $Y \tilde{m} = v Y$ for $\tilde{m} = (0,0,x,y)$
contradicting
$v \not\in \Sp(Y)$.
\item the columns of $N$ are non-zero.  
\end{itemize}
Since $Y$ has the same eigenvalues as its transpose, this implies also that $\det N \neq 0$
and that the rows of $M$ are non-zero.
It is then straightforward to check that these conditions imply the irreducibility of
$M$ under $u$ and $Y$, unless $Y$ has one of the following two forms
$$
Y = -2 \left( \begin{array}{ccccc} \alpha^2 & 0 & x & 0 \\  
0 & (\alpha+1)^2 & 0 & y \\ 
z & 0 & \alpha(\alpha+1) & 0 \\ 
0 & t & 0&\alpha(\alpha+1)  \end{array} \right) 
\ \mbox{ or }\ 
Y = -2 \left( \begin{array}{ccccc} \alpha^2 & 0 & 0 & x \\  
0 & (\alpha+1)^2 & y & 0 \\ 
0 & z & \alpha(\alpha+1) & 0 \\ 
t & 0 & 0&\alpha(\alpha+1)  \end{array} \right) 
\ \ 
$$
But in these two cases, $Y$ would have two distinct eigenvalues whose sum
is $-2(\alpha^2 + \alpha(\alpha+1) )= -2(2 \alpha^2 + \alpha)$ (the two others having
sum $-2((\alpha+1)^2 + \alpha(\alpha+1) )= -2(2 \alpha^2 + 3\alpha+1)$).
It is immediately checked that, when $r \geq 3$, this cannot happen for generic $\alpha$,
and this concludes the proof.
\end{proof}

This theorem implies the following.
\begin{cor} For generic values of $\alpha$, the monodromy morphisms $\C((h)) P_n \to LG_n(\alpha)$ are surjective, where $P_n$ denotes the pure braid group on $n$ strands.
\end{cor}
\begin{proof} The monodromy construction provides a morphism $A P_n \to LG_n(\alpha)$
where $A = \C[[h]]$. The image of $P_n$ is generated by elements whose
image is $1 + 2h \Omega_{ij} + \mbox{ higher terms }$. The conclusion
is then an elementary application of Nakayama's lemma.
\end{proof}

\begin{remark} In the case of ordinary Lie algebras, the list of cases for which
we have a similar property has been obtained in \cite{LEHRERZHANG}. 
\end{remark}

Since, for generic $\alpha \in \mathbbm{k}$, the Bratteli diagram does not depend on $\alpha$, we can define a generic version of the split semisimple (tower of) algebra(s) $LG_n(\alpha)$, defined over an arbitrary field $K$, that we denote $LG_n$. By
the $R$-matrix construction, we get a morphism $K B_n \to LG_n$, with $K$
as in section \ref{sectLG}. By the Kohno-Drinfeld theorem of N. Geer, this morphism
is conjugated, up to a convenient embedding $K \into \C((h))$, to the monodromy
morphisms considered above. This embedding maps
$q \mapsto e^{-h}$, $t_0 \mapsto q^{-2 \alpha}$, $t_1 \mapsto q^{2(\alpha+1)}$, $a \mapsto (-1) q^{2 \alpha(\alpha+1)} = - \exp(-2h \alpha(\alpha+1))$,
$b =-a t_0\mapsto \exp(-2 h \alpha^2)$, $c = -a t_1 \mapsto \exp(-2 h (\alpha+1)^2)$.

The previous corollary thus implies the following one.

\begin{cor} \label{corsurj} The morphism $K P_n \to LG_n$ is surjective, where $P_n$ denotes the pure braid group on $n$ strands.
\end{cor}

{\ \ }

\begin{figure}
\xymatrix{
 &  & \un \ar[d] \\
 &  & V(0,\alpha) \ar[dl] \ar[d] \ar[dr] \\
& V(0,2 \alpha)\ar[dl] \ar[drr] \ar[d] &  V(0,2\alpha+1) \ar[d] \ar[dl] \ar[drr] & V(1,2 \alpha) \ar[dll]\ar[d]\ar[dr] \ar[drr] \\
V(0,3\alpha) & V(0,3\alpha+1) &  V(0,3\alpha+2) & V(1,3\alpha)& V(1,3\alpha+1)&V(2,3\alpha)\\
}

\xymatrix{
 &  & [0,0]_0 \ar[d] \\
 &  & [0,0]_1 \ar[dl] \ar[d] \ar[dr] \\
& [0,0]_2\ar[dl] \ar[drr] \ar[d] &  [0,1]_2 \ar[d] \ar[dl] \ar[drr] & [1,0]_2 \ar[dll]\ar[d]\ar[dr] \ar[drr] \\
{[}0,0]_3 & [0,1]_3 &  [0,2]_3 & [1,0]_3& [1,1]_3&[2,0]_3\\
}
\caption{Generic encoding of the $\sl(2|1)$-modules on the Bratteli diagram for $r \leq 3$.} \label{figbrat3}
\end{figure}
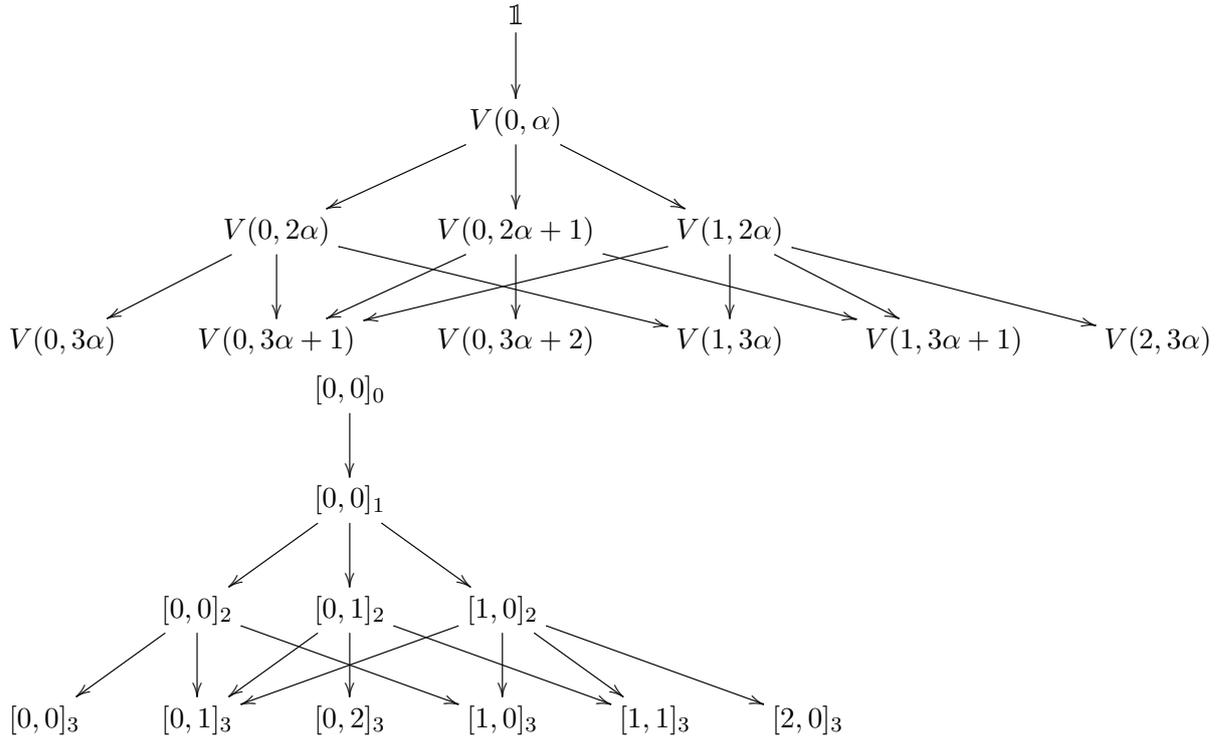

\begin{table}
\begin{center}
\resizebox{14cm}{!}{
\xymatrix{
 &  &  & [0,0]_0 \ar[d] \\
 &  &  & [0,0]_1 \ar[dl] \ar[d] \ar[dr] \\
 & & [0,0]_2\ar[dl] \ar[drr] \ar[d] &  [0,1]_2 \ar[d] \ar[dl] \ar[drr] & [1,0]_2 \ar[dll]\ar[d]\ar[dr] \ar[drr] \\
 & {[}0,0]_3 \ar[drrr] \ar[dl] \ar[d] & [0,1]_3 \ar[dl] \ar[d] \ar[drrr] &  [0,2]_3 \ar[dl] \ar[d] \ar[drrr] & [1,0]_3 \ar[dlll] \ar[d] \ar[dr] \ar[drrr] & [1,1]_3  \ar[dlll] \ar[d] \ar[dr] \ar[drrr] &[2,0]_3 \ar[dl] \ar[dr] \ar[drr] \ar[drrr] \\
 {[}0,0]_4 \ar[dr] \ar[drr] \ar[ddrr]& [0,1]_4 \ar[dr] \ar[drr] \ar[ddrr] & [0,2]_4 \ar[dr] \ar[drr] \ar[ddrr] & [0,3]_4  \ar[dr] \ar[drr] \ar[ddrr] & [1,0]_4 \ar[dll] \ar[ddll] \ar[ddl] \ar[drr] & [1,1]_4  \ar[dll] \ar[ddll] \ar[ddl] \ar[drr]  & [1,2]_4  \ar[dll] \ar[ddll] \ar[ddl] \ar[drr] & [2,0]_4 \ar[ddllll] \ar[dl]
\ar[d] \ar[ddl] & [2,1]_4 \ar[ddllll] \ar[dl]
\ar[d] \ar[ddl] & [3,0]_4 \ar[ddl] \ar[ddll] \ar[ddlll] \ar[dll] \\
 & {[}0,0]_5 & [0,1]_5 & [0,2]_5 & [0,3]_5 & [0,4]_5 & [2,0]_5 & [2,1]_5 & [2,2]_5 \\
  &  & [1,0]_5  & [1,1]_5  &  [1,2]_5 & [1,3]_5 &  [3,0]_5 & [3,1]_5 & [4,0]_5 \\ 
}
}
\end{center}
\caption{Bratteli diagram for $r \leq 5$ : labels}
\label{brattelir5}
\end{table}

\begin{table}

\begin{center}
\mbox{
\xymatrix{
& &  &  & 1 \ar[d] \\
& &  &  & 1 \ar[dl] \ar[d] \ar[dr] \\
& & & 1\ar[dl] \ar[drr] \ar[d] &  1 \ar[d] \ar[dl] \ar[drr] & 1 \ar[dll]\ar[d]\ar[dr] \ar[drr] \\
& & 1 \ar[drrr] \ar[dl] \ar[d] & 3 \ar[dl] \ar[d] \ar[drrr] &  1 \ar[dl] \ar[d] \ar[drrr] & 2 \ar[dlll] \ar[d] \ar[dr] \ar[drrr] & 2  \ar[dlll] \ar[d] \ar[dr] \ar[drrr] & 1\ar[dl] \ar[dr] \ar[drr] \ar[drrr] \\
& 1\ar[dr] \ar[drr] \ar[ddrr]& 6 \ar[dr] \ar[drr] \ar[ddrr] & 6\ar[dr] \ar[drr] \ar[ddrr] & 1 \ar[dr] \ar[drr] \ar[ddrr] & 3 \ar[dll] \ar[ddll] \ar[ddl] \ar[drr] & 8  \ar[dll] \ar[ddll] \ar[ddl] \ar[drr]  & 3  \ar[dll] \ar[ddll] \ar[ddl] \ar[drr] & 3 \ar[ddllll] \ar[dl]
\ar[d] \ar[ddl] & 3 \ar[ddllll] \ar[dl]
\ar[d] \ar[ddl] & 1 \ar[ddl] \ar[ddll] \ar[ddlll] \ar[dll] \\
& & 1 & 10 & 20 & 10 & 1 & 6 & 15& 6 \\
 & &  & 4  & 20  &  20 & 4 &  4 & 4 & 1 \\ 
}
}\end{center}
\caption{Bratteli diagram for $r \leq 5$ :  dimensions}

\end{table}

\section{Cubic Hecke algebras and their representations}
\label{secthecke}

\subsection{Definition and general properties}

Let $K$ be a field of characteristic $0$, and $a,b,c \in K^{\times}$. The cubic Hecke
algebra $H_n = H_n(a,b,c)$ is the quotient of the group algebra $K B_n$ of the braid
group by the relations $(s_i -a)(s_i -b)(s_i -c) = 0$ or, equivalently,
by the single relation $(s_1 -a)(s_1 -b)(s_1 -c) = 0$, since all $s_i$'s
are conjugated in $B_n$. In case $a,b,c$ are three distinct roots of $1$, $H_n$
is the group algebra of the group $\Gamma_n =B_n/s_i^3 = B_n/s_1^3$,
which is known to be finite if and only if $n \leq 5$, by a theorem of Coxeter (see \cite{COX}).

We recall from \cite{G32} the following theorem

\begin{theor} Let $\overline{K}$ denote the algebraic closure of $K$. If $a,b,c$
are algebraically independant over $\Q$, then $H_n \simeq \overline{K} \Gamma_n$
for $n \leq 5$. Moreover, these isomorphisms can be chosen
so that the natural diagrams commute
$$
\xymatrix{
H_2 \ar[r] \ar[d] & H_3 \ar[r] \ar[d]& H_4 \ar[r] \ar[d] & H_5 \ar[d] \\
\overline{K} \Gamma_2 \ar[r] & \overline{K} \Gamma_3 \ar[r] & \overline{K} \Gamma_4 \ar[r]  & \overline{K} \Gamma_5 
}
$$
and they are uniquely defined up to inner automorphism. In particular the correspondence
between irreducible representations of $\overline{K} \Gamma_n$ and $H_n$ is canonical.
\end{theor}

A consequence of this theorem is that the algebras $H_n$ are semisimple,
and thus isomorphic (over $\overline{K}$) to a direct sum
of matrix algebras, each of the matrix algebras corresponding to
an irreductible character of $\Gamma_n$.

Note that, inside $LG_n(\alpha)$, $s_1$ acts as a conjugate of $\tau \exp h \Omega_{12}$, with $\tau, \Omega_{12}$
as in section \ref{sectVV}, and thus it acts semisimply with eigenvalues $b=\exp -2h \alpha^2, c=\exp -2h (\alpha+1)^2, a=- \exp -2h \alpha(\alpha+1)$.
For generic values, $1, \alpha$ and $\alpha^2$ are linearly independent over $\Q$, and thus these eigenvalues are
algebraically independent. We can thus 
consider $LG_n$ as a quotient of $H_n$, that is
$LG_n = H_n/\mathfrak{I}_n$ for an ideal $\mathfrak{I}_n$
over $H_n$. By the above theorem, $\mathfrak{I}_n$ is itself a sum of matrix
algebras, and is uniquely determined by the irreducible representations
of $H_n$ which are not annihilated by $\mathfrak{I}_n$.

Recall from \cite{G32} that $H_3 = H_2 + H_2 s_2 H_2 + H_2 s_2^{-1} H_2 + H_2 s_2 s_1^{-1} s_2$,
that is $H_3 = H_2 + H_2 s_2 H_2 + H_2 s_2^{-1} H_2 + K s_2 s_1^{-1} s_2 + K s_1 s_2 s_1^{-1} s_2+ K s_1^{-1} s_2 s_1^{-1} s_2$.
In particular a Markov trace on $H_3$ is uniquely determined by its value on $H_2$ and its value on 
$s_1^{-1} s_2 s_1^{-1} s_2$, since $tr((s_1 s_2 s_1^{-1}) s_2) = 
tr((s_2 s_1 s_2) s_1^{-1} ) = 
tr(s_1 s_2 s_1 s_1^{-1} ) = tr(s_1 s_2)$. Notice that the value on $s_1^{-1} s_2 s_1^{-1}
s_2$ is the value of the corresponding link invariant
on the figure-eight knot $4_1$.


\subsection{The cubic algebra on $3$ strands and Ishii's relations}

We recall from \cite{COX} that $\Gamma_3$ is a semidirect product $Q_8 \rtimes (\Z/3)$,
where $Q_8$ is a quaternion group of order $8$. Its deformation $H_3$ admits
\begin{itemize}
\item three 1-dimensional representations $S_x$ for $x \in \{a,b,c\}$,
defined by $s_i \mapsto x$. 
\item three 2-dimensional representations $T_{x,y}$
indexed by the subsets $\{x,y \} \subset \{ a,b,c \}$ of cardinality 2, 
\item one 3-dimensional representation $V$.
\end{itemize}
The spectrums of the generators are $\mathrm{Sp}(S_x(s_i)) = \{ x \}$,
$\mathrm{Sp}(T_{x,y}(s_i)) = \{ x,y \}$, $\mathrm{Sp}(V(s_i)) = \{ a,b,c \}$. In particular,
these irreducible representations of $B_3$ (or $H_3$) can be uniquely identified
through their restriction to $B_2$. One has matrix models
over $K$ of these irreducible representations (see below), hence $H_3$ is
split semisimple,
and we have 
$$H_3 \simeq K \times K \times K \times M_2(K) \times M_2(K) \times M_2(K) \times M_3(K)
$$
and every ideal of $H_3$ is a sum of some of these matrix algebras. For instance,
the ideal corresponding to the $BMW$ algebra has the form $K \times M_2(K)
\times M_3(K)$, whereas the ideal corresponding to the Funar-Bellingeri
quotient of \cite{BELLFUN} (see also \cite{CABANESMARIN}, \cite{G32}) is $Z(H_3) = K \times K \times K$. The mere fact that $\dim LG_3 = 20 = 24 - 2^2$
imposes that the ideal defining $LG_3$ as a quotient of $H_3$ is one of the $M_2(K)$,
precisely the only one which does not correspond to a representation of $LG_3$. 
From the Bratteli diagram of $LG_3(\alpha)$ we get that our choice
of parameters imposes that it is
$T_{b,c}$.
Since it is a simple ideal, it is clearly generated by either one of the non-trivial Ishii's relation
of \cite{Ishii3}. The fact
that at least one of Ishii's relation is non trivial in $H_3$ is easily checked,
and can be used to provide another quick proof of theorem 1 of \cite{Ishii3}.
We choose one of these non-trivial relations and call it $r_2$. 

\subsection{Description of the representations of $H_4$}

A description of the irreducible representations of $H_4$ can
be found in \cite{THESE}. We use the same notation here.
There are
\begin{itemize}
\item three 1-dimensional representations $S_x$ for $x \in \{a,b,c\}$,
defined by $s_i \mapsto x$. 
\item three 2-dimensional representations $T_{x,y}$
indexed by the subsets $\{x,y \} \subset \{ a,b,c \}$ of cardinality 2,
which factorize through the special morphism $B_4 \to B_3$ (hence through $H_3$).
\item one 3-dimensional representation $V$, factorizing through $B_3$.
\item six 3-dimensional representations $U_{x,y}$ for each tuple $(x,y)$
with $x \neq y$ and $x,y \in \{a,b,c \}$. 
\item six 6-dimensional representations $V_{x,y,z}$ for each tuple
$(x,y,z)$ with $\{x,y,z \} = \{ a,b,c \}$
\item three 8-dimensional representations $W_x$ for $x \in \{ a,b,c \}$
\item two 9-dimensional representations $X$, $X'$.
\end{itemize}

Except for $X,X'$, they are uniquely defined by their restriction
to $B_3$. If one let $S_x, T_{x,y}, V$ also denote the (irreducible) restriction
to $B_4$ of the representations $S_x, T_{x,y}, V$, we have
$$
\begin{array}{lcl}
\Res U_{x,y} & = & S_x + T_{x,y} \\
\Res V_{x,y,z} &=& S_x + T_{x,y} + V \\
\Res W_x &=& S_x + T_{x,y} + T_{x,z} + V \\  
\Res X &= & T_{x,y} + T_{x,z} + T_{y,z} + V
\end{array}
$$
The representations $T_{x,y}$ of $B_3$ are well-determined by their
restriction to $B_2$ : with obvious notations, $Res_{B_2} T_{x,y} = S_x + S_y$.

A complete set of matrices for these representations was first found by Brou\'e and Malle in \cite{BROUEMALLE}.
Other constructions were subsequently given, in \cite{THESE} and \cite{MALLEMICHEL}.
The latter ones have been included in the development version of the CHEVIE package for GAP3,
and the order in which they are stored in this package at the present time is
$S_a,S_c,S_b,T_{b,c}, T_{a,b}, T_{a,c},V$, $U_{b,a},U_{a,c},U_{c,b},U_{c,a},
U_{a,b}, U_{b,c}$,
$V_{c,a,b}, V_{b,c,a}, V_{a,b,c},V_{b,a,c},V_{c,b,a},
V_{a,c,b}, W_a, W_c, W_b, X , X'.
 $

For the convenience of the reader, we provide these models in tables \ref{tab:reps25a} and \ref{tab:reps25b}.

\begin{table}
$$
T_{b,c} : s_1 \mapsto \begin{pmatrix} b & 0 \\ bc & c \end{pmatrix} 
s_2 \mapsto \begin{pmatrix} c & -1 \\ 0 & b \end{pmatrix} 
s_3 \mapsto \begin{pmatrix} b & 0 \\ bc & c \end{pmatrix} 
$$
{}
$$
V : 
s_1 \mapsto \begin{pmatrix} c & 0 & 0  \\ ac+b^2 & b & 0 \\ b & 1 & a \end{pmatrix} 
s_2 \mapsto \begin{pmatrix} a & -1 & b  \\ 0 & b & -ac-b^2 \\ 0 & 0 & c \end{pmatrix} 
s_3 \mapsto \begin{pmatrix} c & 0 & 0  \\ ac+b^2 & b & 0 \\ b & 1 & a \end{pmatrix} 
$$
{}
$$
U_{a,b} : 
s_1 \mapsto \begin{pmatrix} b & 0 & 0  \\ 0 & a & 0 \\ a & 0 & a \end{pmatrix} 
s_2 \mapsto \begin{pmatrix} a & -a+b & -b  \\ 0 & b & -b \\ 0 & 0 & a \end{pmatrix} 
s_3 \mapsto \begin{pmatrix} a & 0 & 0  \\ 0 & a & 0 \\ -a & 2a & b \end{pmatrix} 
$$
\caption{Representations of $H_4$, first part.}
\label{tab:reps25a}
\end{table}
{}

\begin{table}
$
\begin{array}{c|c|c}
  V_{a,b,c} & W_a & X  \\
\hline
 \resizebox{4cm}{2cm}{$\begin{pmatrix} a & 0 & 0 & 0 & 0 & 0 \\ ac+b^2 & b & 0 & 0 &0 &0 \\ b & 1 & c &0&0&0 \\
0 & 0 & 0 & a & 0 & 0 \\ 0 & 0 & 0 & 0 & a & 0 \\ 0 & 0 & 0 & 1 & 0 & b \end{pmatrix}$} & \resizebox{5cm}{2cm}{$\begin{pmatrix}  b & 0 & 0 & 0 & 0 & 0 & 0 & 0 \\  1 & c & 0 & 0 & 0 & 0
 & 0 & 0 \\  c & ab+c^2 & a & 0 & 0 & 0 & 0 & 0 \\  0 & 0 & 0 & b & 0 & 0 
& 0 & 0 \\  -1 & -c & 0 & b & a & 0 & 0 & 0 \\  0 & 0 & 0 & 0 & 0 & c & 0 
& 0 \\  0 & 0 & 0 & b & 0 & -1 & a & 0 \\  0 & 0 & 0 & 0 & 0 & b-c & 0 & a
 \\ \end{pmatrix}$} & \resizebox{5cm}{2cm}{$\begin{pmatrix}  c & 0 & 0 & 0 & 0 & 0 & 0 & 0 & 0 \\  ac+b^2 & b & 0 & 
0 & 0 & 0 & -j^2bc & 0 & 0 \\  b & 1 & a & 0 & 0 & 0 & c & 0 & 0 \\  0 & 
0 & 0 & a & 0 & 0 & -c & jc & a+j^2b \\  j^2a-b & 0 & 0 & 0 & b & 0 & 0 &
 0 & 0 \\  j^2a & 0 & 0 & 0 & b & a & 0 & 0 & 0 \\  0 & 0 & 0 & 0 & 0 & 0
 & c & 0 & 0 \\  0 & 0 & 0 & 0 & 0 & 0 & 0 & c & 0 \\  0 & 0 & 0 & 0 & 0 &
 0 & 0 & j^2c & b \\ \end{pmatrix}$}  \\
 \resizebox{4cm}{2cm}{$\begin{pmatrix} c & -1 & b & -1 & -b & a  \\ 0 & b & -ac-b^2 & 0 & 0 & -ab \\ 
0 & 0 & a & 0 & 0 & 0 \\ 0 & 0 & 0 & b & ac + b^2 & -ab \\ 0 & 0 & 0 & 0 & a & 0 \\
0 & 0 & 0 & 0 & 0 & a \end{pmatrix}$}
 &\resizebox{5cm}{2cm}{$\begin{pmatrix}  a & -ab-c^2 & c & 0 & 0 & 0 & 0 & 0 \\  0 & c & -1 & 0 
& 0 & 0 & 0 & 0 \\  0 & 0 & b & 0 & 0 & 0 & 0 & 0 \\  0 & -a & 0 & a & -a 
& 0 & 0 & 0 \\  0 & 0 & 0 & 0 & b & 0 & 0 & 0 \\  0 & 0 & -a & 0 & -ac & a
 & ac & 0 \\  0 & 0 & -1 & 0 & b-c & 0 & c & 0 \\  0 & 0 & -c & 0 & bc-c^2
 & 0 & ab+c^2 & a \\ \end{pmatrix}$} & \resizebox{5cm}{2cm}{$\begin{pmatrix}  a & -1 & b & -jb & 0 & 0 & 0 & 0 & b \\  0 & b & -ac-b
^2 & -ac+jb^2 & 0 & 0 & 0 & 0 & -ac-b^2 \\  0 & 0 & c & 0 & 0 & 0 & 0 & 0 &
 0 \\  0 & 0 & 0 & c & 0 & 0 & 0 & 0 & 0 \\  0 & 0 & a & a & a & -a & 0 & 
0 & 0 \\  0 & 0 & jb & 0 & 0 & b & 0 & 0 & 0 \\  0 & 0 & 0 & a & 0 & 0 & 
a & a & a \\  0 & 0 & 0 & 0 & 0 & 0 & 0 & b & -jb \\  0 & 0 & 0 & 0 & 0 &
 0 & 0 & 0 & c \\ \end{pmatrix}$}  \\
 \resizebox{4cm}{2cm}{$\begin{pmatrix} a & 0 & 0 & 0 & 0 & 0   \\ 0 & a & 0 & 0 & 0 & 0 \\
0 & 0 & a & 0 & 0 & 0 \\ ac+b^2 & 0 & 0 & b & 0 & 0 \\ -b & 0 & 0 & -1 & c & 0 \\
0 & 1 & 0 & 0 & 0 & b
\end{pmatrix}$} &\resizebox{5cm}{2cm}{$\begin{pmatrix} a & 0 & 0 & bc-c^2 & 0 & 0 & 0 & 0 \\  0 & a & 0 & c & 
0 & -1 & 0 & 0 \\  0 & 0 & a & 0 & 0 & b-c & ac+c^2 & -c \\  0 & 0 & 0 & c
 & 0 & 0 & 0 & 0 \\  0 & 0 & 0 & 0 & a & 1 & -a & 0 \\  0 & 0 & 0 & 0 & 0 
& b & 0 & 0 \\  0 & 0 & 0 & 0 & 0 & 0 & c & -1 \\  0 & 0 & 0 & 0 & 0 & 0 &
 0 & b \\ \end{pmatrix}$} & \resizebox{5cm}{2cm}{$\begin{pmatrix} c & 0 & 0 & 0 & 0 & 0 & 0 & 0 & 0 \\  ac+b^2 & b & 0 & 
0 & -j^2bc & 0 & 0 & 0 & 0 \\  0 & 0 & b & 0 & 0 & -j^2c & 0 & 0 & 0 \\ 
 0 & 0 & a+j^2b & a & -c & -jc & 0 & 0 & 0 \\  0 & 0 & 0 & 0 & c & 0 & 0 &
 0 & 0 \\  0 & 0 & 0 & 0 & 0 & c & 0 & 0 & 0 \\  j^2a-b & 0 & 0 & 0 & 0 &
 0 & b & 0 & 0 \\  -j^2a & 0 & 0 & 0 & 0 & 0 & -b & a & 0 \\  b & 1 & 0 &
 0 & c & 0 & 0 & 0 & a \\ \end{pmatrix}$}  \\
 \end{array}
 $
 \begin{center}
 Images of $s_1,s_2,s_3$ (from top to bottom) under $V_{a,b,c}$, $W_a$ and $X$. \\
 In these formulas, $j$ denotes  a third root of 1.
 \end{center}
\caption{Representations of $H_4$, second part. }
\label{tab:reps25b}
\end{table}

\subsection{Representations of $H_4$ which factorize through $LG_4$}
\label{soussectchar}

Since $LG_3$ is the quotient of $H_3$ by the ideal corresponding to its
representation $T_{b,c}$,
It follows that the quotient of the group algebra of $B_4$
by the cubic relation and $(r_2)$ can be identified with
the quotient of $H_4$
by the ideal $J$ corresponding to
all irreducible representations whose restriction to $H_3$
contains an irreducible component of type $T_{b,c}$. Viewed the other
way round, this proves that the quotient algebra 
$H_4/J$
is a semisimple
algebra whose irreducible representations are the
irreducible representations of $H_4$ whose restriction to
$H_3$
do \emph{not} contain an irreducible component of type $T_{b,c}$.

From the description of the branching rule we get that these
representations are the $S_x$ for $x \in \{a,b,c \}$, $T_{a,b}$, $T_{a,c}$,
$V$, $U_{a,b}$, $U_{b,a}$, $U_{a,c}$, $U_{c,a}$,
$V_{a,b,c}$, $V_{b,a,c}$, $V_{a,c,b}$, $V_{c,a,b}$, $W_a$,
hence the corresponding quotient has dimension $1^2 \times 3 + 2^2 \times 2 + 3^2+ 3^2 \times 4+6^2 \times4+8^2 \times 1 =264$.
From the Bratteli diagrams of section \ref{sectstruct} we can compute $\dim LG_n$ for
arbitrary $n$, and we get that $LG_4$ has dimension 175. This proves that
\emph{Ishii's relations are not sufficient} to define $LG_n$. More precisely :

\begin{prop} The quotient of $K B_4$ by a generic cubic relation $r_1$ and
Ishii's relation $r_2$ has dimension
$\dim K B_4/(r_1,r_2) = \dim H_4/(J) = 264 > \dim LG_4 = 175$.
\end{prop}

\begin{remark} {\ } \\
\begin{enumerate}
\item The following argument, by contradiction, could alternatively be used to show that $K B_4/(J) \neq LG_4 $. Up to some
permutation of the parameters, $BMW_3$ is a quotient of $LG_3$. Since the relations of $BMW_n$ are on 3 strands, 
this implies that $BMW_4$ is a quotient of $KB_4/(J)$. But it is not a quotient of $LG_4$, and actually even the (ordinary, quadratic) Hecke algebra on
4 strands is not a quotient of $LG_4$ ; this can indeed be inferred from the fact that the Bratteli diagram for the representation of the Hecke
algebra traditionnally indexed by the partition $[2,2]$ does not appear as a sub-diagram inside the Brattelli diagram of $LG_4$.
\item By the same method, one can get the dimension of $K B_5/(r_1,r_2)$. It is $6490$.
\end{enumerate}\end{remark}

We now explain how to get a matrix description of $LG_4$. Since
each irreducible representation of $H_4$
is well-determined from
its restriction to $H_3$,
the Bratteli diagram obtained before
determines the irreducible representation of $LG_4$ ; they
are the
$S_a, S_b, S_c$,
$U_{b,a}, U_{a,c},
U_{c,a},U_{a,b}, V_{c,a,b}, V_{b,a,c}, W_a$.
This provides an explicit morphism 
$\Phi : K B_4 \onto K^3  \times M_3(K)^4 \times 
M_6(K)^2 \times M_8(K) \simeq K^{175}$
factorizing though $H_4$
whose image can be identified with $LG_4$.

\xymatrix{
 &  &  & 1\ar[d] \\
 &  &  & 1 \ar[dl] \ar[d] \ar[dr] \\
 & & S_c \ar[dl] \ar[drr] \ar[d] &  S_b \ar[d] \ar[dl] \ar[drr] & S_a \ar[dll]\ar[d]\ar[dr] \ar[drr] \\
 & S_c \ar[drrr] \ar[dl] \ar[d] & V \ar[dl] \ar[d] \ar[drrr] &  S_b \ar[dl] \ar[d] \ar[drrr] & T_{a,c} \ar[dlll] \ar[d] \ar[dr] \ar[drrr] & T_{a,b}  \ar[dlll] \ar[d] \ar[dr] \ar[drrr] &S_a \ar[dl] \ar[dr] \ar[drr] \ar[drrr] \\
  S_c & V_{c,a,b}   & V_{b,a,c}  & S_b  & U_{c,a}  & W_a  & U_{b,a}   & U_{a,c} 
 & U_{a,b} 
 & S_a  \\
}

In the forecoming section we will use this matrix description to get a inductive properties of this algebra. Before
that, we conclude the present section by describing the part of $LG_n$ which factorizes through ordinary quadratic Hecke
algebras.

There is clearly one quotient of $LG_n$ which factors through a quadratic Hecke algebra for each
irreducible component $[u,v]_2$ of $V(0,\alpha)^{\otimes 2}$ : if $\langle u,v \rangle_2$ denotes
the simple ideal of $LG_2$ corresponding to it, then $LG_n/\langle u,v \rangle_2$ is a quotient
of a quadratic Hecke algebra. Any quotient of a quadratic (generic) Hecke algebra is uniquely determined by its
irreducible representations, namely the irreducible representations of the Hecke algebra (indexed as usual by
partitions of $n$) which factor through it. It follows that  this quotient can be uniquely identified by its Bratteli diagram, which is
the part of the Bratteli diagram of $LG_n$ made of the paths which do \emph{not} pass through $[u,v]_2$. 
When $[u,v]_2 = [1,0]_2$ we easily get from this description that the irreducible representations of $LG_n/ \langle 1, 0 \rangle_2$ are the 1-dimensional
representations corresponding to $[0,0]_n$ and $[0,n-1]_n$, hence $LG_n/ \langle 1,0 \rangle_2 \simeq K^2$. Under
the correspondence between irreducible representations of the quadratic Hecke algebras and partitions,
which is canonical up to the operation of transposing all Young diagrams at the same time,
we have for instance $[0,0]_n \mapsto [n]$,  $[0,n-1]_n \mapsto [1^n]$. When $[u,v]_2 = [0,0]_2$, we get that
the irreducible representations of $LG_n/ \langle 0,0 \rangle_2$ are the $[n-m,m]_n$ for $1 \leq m \leq n$,
and that the correspondence with partitions can be chosen to be $[n-m,m]_n \mapsto [n-m+1,1^{m-1}]$.
The situation is similar with $[u,v]_2 = [0,1]_2$, the correspondence being given by $[m,0]_n \mapsto [n-m,1^m]$. Since
the quotient of the quadratic Hecke algebra whose irreductible representations are indexed by
hook partitions, which is a centralizer algebra for $\mathfrak{gl}(1|1)$ inside $\End(U^{\otimes n})$
with $U$ its standard 2-dimensional representation (see e.g.\cite{PLANAR}), is a defining algebra
for the Alexander polynomial (see \cite{ROZ,VIR}), this gives another explanation for the well-known connection between
the Links-Gould and Alexander polynomials (see \cite{ISHIIALEX,PATGEER}).

\section{Inductive description of $LG_n$}

We investigate the image of braid words in the algebras $LG_3$ and $LG_4$. We first determine a basis made of braid
words for $LG_3$ (section \ref{sect6LG3}), then of $LG_4$ (section \ref{sect6LG4}). This first basis enables us to get a new relation in $LG_4$.
We now start again from the basis of $LG_3$ obtained earlier and get in section \ref{sect6LG4bis} a new basis which is a more
suitable for induction. We use it to prove a decomposition of $LG_4$ as a $LG_3$-bimodule, and then more generally,
in section \ref{sect6bimod}, a decomposition of $LG_{n+1}$ as a $LG_n$-bimodule.

\subsection{A basis for $n=3$}
\label{sect6LG3}

The braid words with at most 3 crossings and avoiding the
pattern $s_i^r$ for $|r| \geq 2$ are $1$, $s_1^{\pm 1}$, $s_2^{\pm 1}$,
$s_1^{\pm 1} s_2^{\pm 1}$, $s_2^{\pm 1} s_1^{\pm 1}$, $s_1^{\pm 1} s_2^{\pm 1}
s_1^{\pm 1}$, $s_2^{\pm 1} s_1^{\pm 1} s_2^{\pm 1}$,
that is  $29$ words, whose image span $LG_3$. Among them, there
are 13 words ($13 = 1 + 4 + 2 \times 2 + 2 \times 2$) with at most two crossings.
These 13 words have linearly independent images in $LG_3$.
We denote $\mathcal{B}_0 = \mathcal{B}_0^{(3)}$ these images, $V_0$ the subspace of $LG_3$
that they span (see table \ref{tab:b3twocross}).
$$\mathcal{B}_0 = [1, s_1, s_1^{-1}, s_2, s_2^{-1}, s_1s_2, s_1s_2^{-1}, 
  s_1^{-1}s_2, s_1^{-1}s_2^{-1}, s_2s_1, s_2s_1^{-1}, s_2^{-1}s_1, 
  s_2^{-1}s_1^{-1}]$$
  

We let $V_1$ denote the span of the 8 words of the
form $s_1^{\pm 1} s_2^{\pm 1} s_1^{\pm 1}$. We have $\dim V_0 + V_1 = 19$.
Somewhat arbitrarily, we choose for basis of $LG_3$ the following set (see table \ref{tab:baseBpour3})
$$
\mathcal{B}^{(3)} = \mathcal{B} = \mathcal{B}_0 \sqcup   [  s_1^{-1}s_2^{-1}s_1^{-1}, s_1^{-1}s_2^{-1}s_1, 
  s_1^{-1}s_2s_1^{-1}, s_1s_2^{-1}s_1^{-1}, s_1s_2^{-1}s_1, 
  s_1s_2s_1^{-1}, s_2^{-1}s_1s_2^{-1} ];
$$

\begin{table}
\hrule

\begin{tikzpicture}
\braid[braid colour=red,strands=3,braid start={(0,0)}]%
{ \dsigma _1  }
\node[font=\Huge] at (3.5,-1) {\(,\)};
\braid[braid colour = red,strands=3,braid start={(3.5,0)}]
{\dsigma_1^{-1} }
\node[font=\Huge] at (7,-1) {\(,\)};
\braid[braid colour = red,strands=3,braid start={(7,0)}]
{\dsigma_2 }
\node[font=\Huge] at (10.5,-1) {\(,\)};
\braid[braid colour = red,strands=3,braid start={(10.5,0)}]
{\dsigma_2^{-1} }
\node[font=\Huge] at (14,-1) {\(,\)};
\end{tikzpicture}

\hrule

\begin{tikzpicture}
\braid[braid colour=red,strands=3,braid start={(0,0)}]%
{ \dsigma _1  \dsigma_2}
\node[font=\Huge] at (3.5,-1) {\(,\)};
\braid[braid colour = red,strands=3,braid start={(3.5,0)}]
{\dsigma_1 \dsigma_2^{-1} }
\node[font=\Huge] at (7,-1) {\(,\)};
\braid[braid colour = red,strands=3,braid start={(7,0)}]
{\dsigma_1^{-1} \dsigma_2 }
\node[font=\Huge] at (10.5,-1) {\(,\)};
\braid[braid colour = red,strands=3,braid start={(10.5,0)}]
{\dsigma_1^{-1} \dsigma_2^{-1} }
\node[font=\Huge] at (14,-1) {\(,\)};
\end{tikzpicture}

\hrule

\begin{tikzpicture}
\braid[braid colour=red,strands=3,braid start={(0,0)}]%
{ \dsigma _2  \dsigma_1}
\node[font=\Huge] at (3.5,-1) {\(,\)};
\braid[braid colour = red,strands=3,braid start={(3.5,0)}]
{\dsigma_2 \dsigma_1^{-1} }
\node[font=\Huge] at (7,-1) {\(,\)};
\braid[braid colour = red,strands=3,braid start={(7,0)}]
{\dsigma_2^{-1} \dsigma_1 }
\node[font=\Huge] at (10.5,-1) {\(,\)};
\braid[braid colour = red,strands=3,braid start={(10.5,0)}]
{\dsigma_2^{-1} \dsigma_1^{-1} }
\node[font=\Huge] at (14,-1) {};
\end{tikzpicture}

\hrule
\caption{$\mathcal{B}_0$ : 3-braids with at most 2 crossings.}
\label{tab:b3twocross}
\end{table}

\begin{table}
\resizebox{15cm}{!}{
\begin{tikzpicture}
\braid[braid colour=red,strands=3,braid start={(0,0)}]%
{ \dsigma _1^{-1} \dsigma_2^{-1} \dsigma_1^{-1}   }
\node[font=\Huge] at (3.5,-1) {\(,\)};
\braid[braid colour = red,strands=3,braid start={(3.5,0)}]
{\dsigma_1^{-1}  \dsigma_2^{-1} \dsigma_1 }
\node[font=\Huge] at (7,-1) {\(,\)};
\braid[braid colour = red,strands=3,braid start={(7,0)}]
{\dsigma_1^{-1} \dsigma_2 \dsigma_1^{-1} }
\node[font=\Huge] at (10.5,-1) {\(,\)};
\braid[braid colour=red,strands=3,braid start={(10.5,0)}]%
{ \dsigma _1 \dsigma_2^{-1} \dsigma_1^{-1}   }
\node[font=\Huge] at (14,-1) {\(,\)};
\braid[braid colour = red,strands=3,braid start={(14,0)}]
{\dsigma_1  \dsigma_2^{-1} \dsigma_1 }
\node[font=\Huge] at (17.5,-1) {\(,\)};
\braid[braid colour = red,strands=3,braid start={(17.5,0)}]
{\dsigma_1  \dsigma_2 \dsigma_1^{-1} }
\node[font=\Huge] at (21,-1) {\(,\)};
\braid[braid colour=blue,strands=3,braid start={(21,0)}]%
{ \dsigma _2^{-1}  \dsigma_1  \dsigma_2^{-1}   }
\end{tikzpicture}}
\caption{The $6+1$ braid words added to $\mathcal{B}_0$ to form $\mathcal{B}$}
\label{tab:baseBpour3}
\end{table}

This basis is
made of all possible braid words with at most 2 crossings
with in addition 6 of the 8 words of the form $s_1^{\pm 1} s_2^{\pm 1} s_1^{\pm 1}$
(the words $s_1 s_2 s_1$ and $s_1^{-1} s_2 s_1$ are removed), which together
form a basis of $V_0 + V_1$,
and one of the form $s_2^{\pm 1} s_1^{\pm 1} s_2^{\pm 1}$. 
Since, according to \cite{G32}, the words $s_1^{\pm} s_2^{\pm} s_1^{\pm}$ are linearly
independent in $H_3$,
there should be a relation 
arising from the expression of
$s_1s_2s_1$ and  $s_1^{-1} s_2 s_1$.

We get the following :

$$
\begin{array}{lcl}
s_1^{-1} s_2 s_1 &=& 
\frac{-1}{a}  s_1 s_2  
+ a s_1 s_2^{-1}  + a s_1^{-1} s_2  - 
a^3  s_1^{-1} s_2^{-1}  +a^{-1} s_2 s_1  - a s_2 s_1^{-1}  -
a s_2^{-1} s_1  +  a^3 s_2^{-1} s_1^{-1} \\
& &  + a^2 s_1^{-1} 
 s_2^{-1} s_1  - a^2 s_1 s_2^{-1} s_1^{-1}  +  s_1 s_2 s_1^{-1}
 \end{array}
$$
so, in particular,
$
s_1^{-1} s_2 s_1 \equiv  s_1 s_2 s_1^{-1}  +  a^2 (s_1^{-1} 
 s_2^{-1} s_1  -  s_1 s_2^{-1} s_1^{-1})   \mod V_0
$,
that is
\begin{center}
\resizebox{!}{2cm}{
\begin{tikzpicture}
\braid[braid colour=red,strands=3,braid start={(0,0)}]%
{ \dsigma _1 ^{-1} \dsigma_2 \dsigma_1  }
\node[font=\Huge] at (3.5,-1) {\(\equiv \)};
\braid[braid colour = red,strands=3,braid start={(3.5,0)}]
{\dsigma_1 \dsigma_2 \dsigma_1^{-1} }
\node[font=\Huge] at (7.5,-1) {\( + a^2  \)};
\braid[braid colour = red,strands=3,braid start={(7.5,0)}]
{\dsigma_1^{-1} \dsigma_2^{-1} \dsigma_1}
\node[font=\Huge] at (11.5,-1) {\( - a^2  \)};
\braid[braid colour = red,strands=3,braid start={(11.5,0)}]
{\dsigma_1 \dsigma_2^{-1} \dsigma_1^{-1}}

\end{tikzpicture}} 
\end{center}
We also have
$$
\begin{array}{lcl}
s_1 s_2 s_1 &=& 
\left( -ab-ac-bc-a^2\right) s_2  + \left( a^2bc+a^3b+a^3c+a^4\right) s_2^{-1}  + \left( a+b+c\right) s_1 s_2  \\
& & + \left( -a^2b-a^2c-a^3\right) s_1 s_2^{-1}  + \left( abc\right) s_1^{-1} s_2  +
 \left( -a^3bc\right) s_1^{-1} s_2^{-1}  + \left( a\right) s_2 s_1  \\
 & & + \left( abc+a^2b+a^2c\right) s_2 s_1^{-1}  + \left( -a^3\right) s_2^{-1} s_1  + \left( -a^3bc-a^4b-a^4c\right) s_2^{-1} 
s_1^{-1} \\
& &  + \left( a^4bc\right) s_1^{-1} s_2^{-1} s_1^{-1}  + \left( -a^2bc
\right) s_1^{-1} s_2 s_1^{-1}  + \left( a^3b+a^3c\right) s_1 s_2^{-1} s_1^{-1} \\
& & 
  + \left( a^2\right) s_1 s_2^{-1} s_1  + \left( -ab-ac\right) s_1 s_2 s_1^{\
-1} 
\end{array}
$$
hence
$$
s_1 s_2 s_1 \equiv 
  a^3(abc) s_1^{-1} s_2^{-1} s_1^{-1}  - a(abc)
s_1^{-1} s_2 s_1^{-1}  +a^3(b+c) s_1 s_2^{-1} s_1^{-1} \\
  +  a^2 s_1 s_2^{-1} s_1  -a(b+c)  s_1 s_2 s_1^{\
-1}  \mod V_0
$$

\subsection{Case $n=4$ : new relations over 4 crossings}
\label{sect6LG4}

\begin{table}
\begin{center}
\resizebox{!}{6cm}{\includegraphics{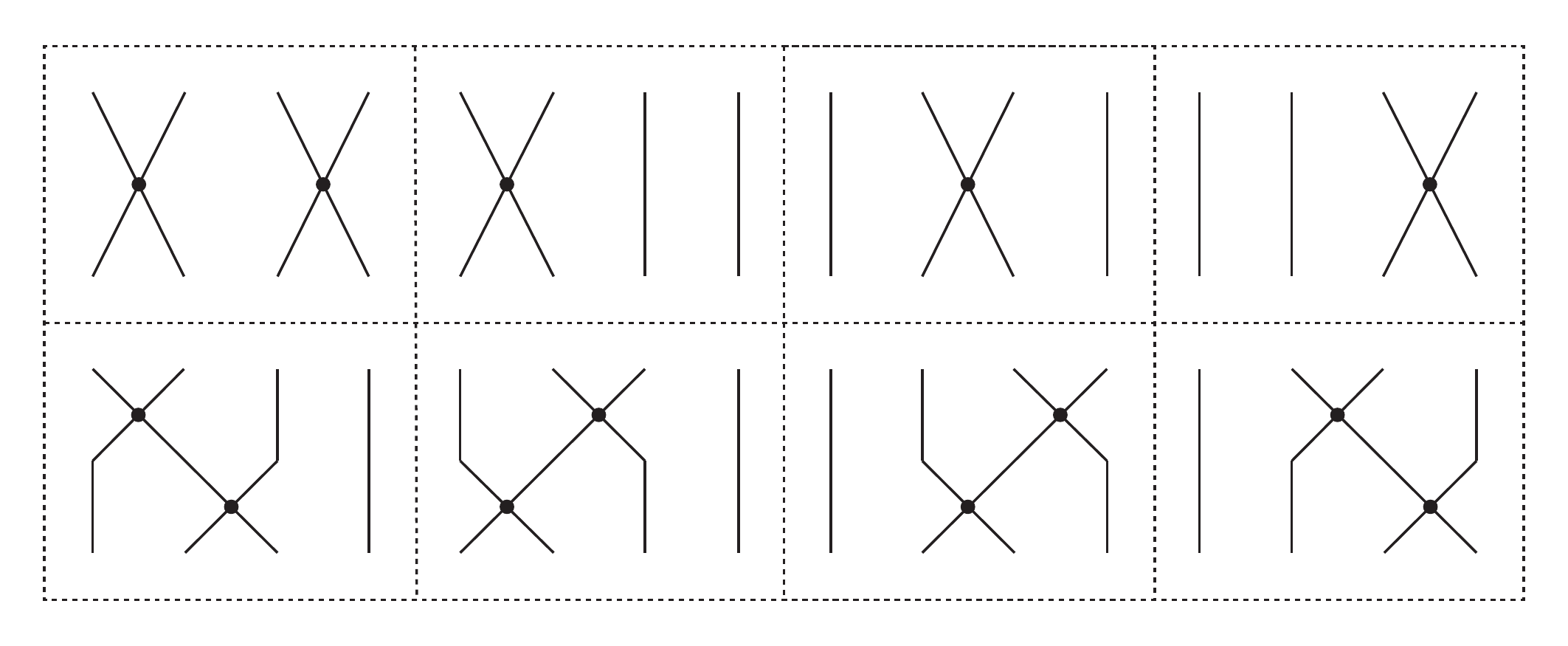}}
\caption{Patterns for 1 and 2 crossings, on 4 strands}
\label{tab:B04}
\end{center}
\end{table}

We denote $\mathcal{B}_0^{(4)}$  the set of $1 + 2 \times 3 + 2^2 \times 5 = 27$ braids with at most 2 crossings,
which correspond to the patterns described in table \ref{tab:B04} (plus the trivial braid, with 0 crossings),
$\mathcal{B}_1^{(4)}$ the set of braids with 3 crossings described by the patterns of table
\ref{tab:B14}. Because of the study of $LG_3$ and of
the cubic relation, the image in $LG_4$ of every braid with at most 3 crossings
can be written as a linear combination of $\mathcal{B}_0^{(4)} \sqcup \mathcal{B}_1^{(4)}$.
We check that these $8 \times 4 + 2 \times 7 + 27 =  73$ elements are indeed linearly independent in $LG_4$, and we
let $V_3$ denote the subspace they span.

\begin{table}
\begin{center}
\resizebox{!}{6cm}{\includegraphics{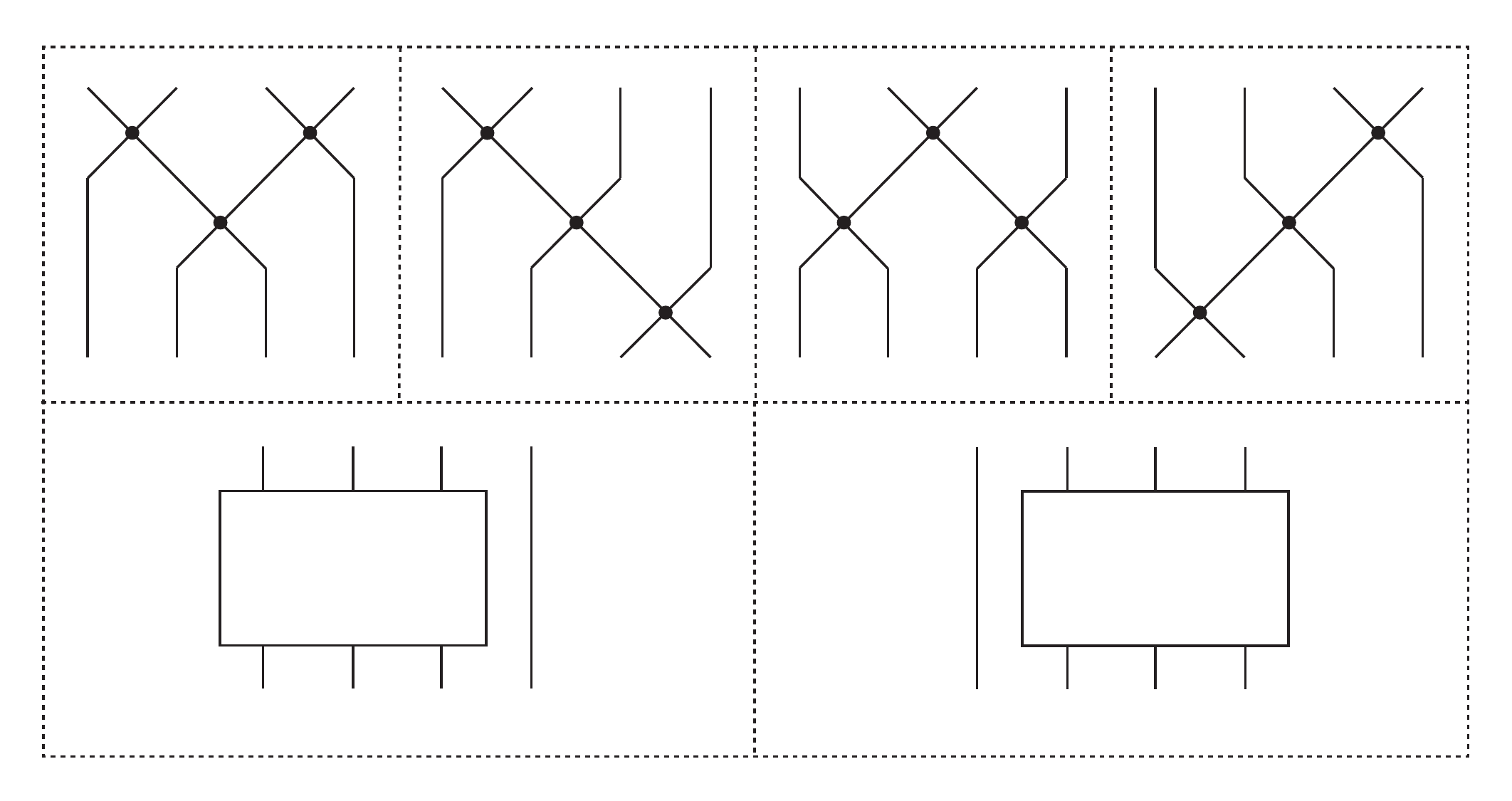}}
\end{center}
\begin{center}
The box represents arbitrary elements of $\mathcal{B}^{(3)}\setminus \mathcal{B}_0^{(3)}$.
\end{center}
\caption{Patterns for 3 crossings, on 4 strands. }
\label{tab:B14}
\end{table}

\begin{table}
\begin{center}
\resizebox{16cm}{!}{\includegraphics{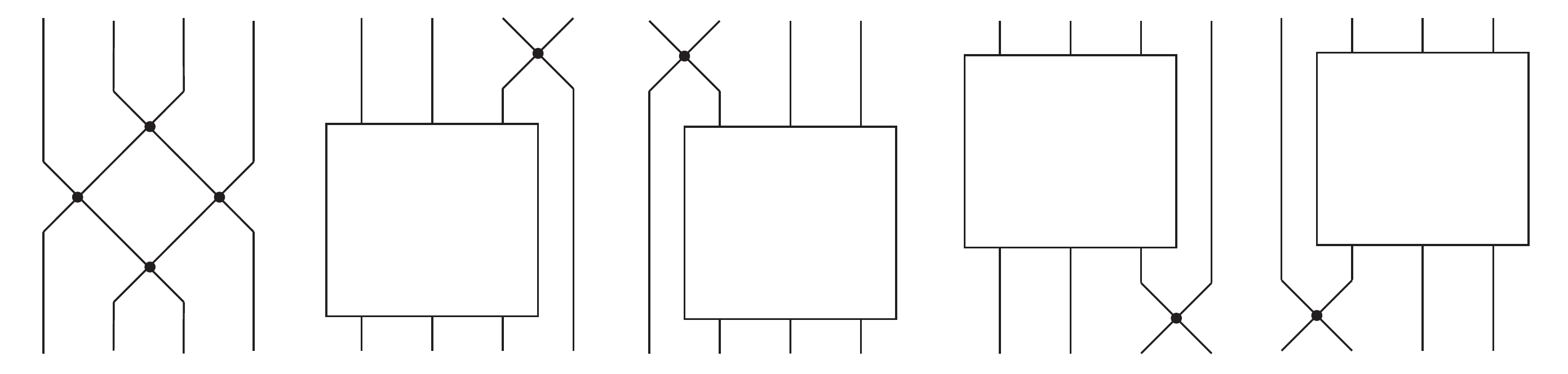}}
\end{center}
\begin{center}
The box represents arbitrary elements of $\mathcal{B}^{(3)}\setminus \mathcal{B}_0^{(3)}$.
\end{center}
\caption{Patterns for 4 crossings, on 4 strands. }
\label{tab:B24}
\end{table}

We denote $\mathcal{B}_2^{(4)}$  the set of $2^4 + 7 \times 2 \times 4 = 72$ braids with 4 crossings,
which correspond to the patterns described in table \ref{tab:B24}. Their images in $LG_4$
are linearly independent, that is they span a subspace $V_4$ of dimension $72$. We
get $\dim (V_3 + V_4) = 141 < 145 = 72+73$.

Among these 145 words in $s_i$ and $s_i^{-1}$, only 5 of them
contain a pattern $s_3^{\pm} s_2^{\pm} s_3^{\pm}$. Besides $s_3^{-1} s_2 s_3^{-1}$, they
are the $\mathcal{A}_{3,U}^{(4)} = \{ s_1^{\eps} s_3^{-1} s_2 s_3^{-1} \ | \ \eps \in \{ -1 , 1 \}\}$
and 
$\mathcal{A}_{3,D}^{(4)} = \{  s_3^{-1} s_2 s_3^{-1} s_1^{\eps} \ | \ \eps \in \{ -1 , 1 \}\}$.

Let $\mathcal{B}_3^{(4)} = (\mathcal{B}^{(4)}_0 \sqcup \mathcal{B}^{(4)}_1 \sqcup
\mathcal{B}^{(4)}_2) \setminus (\mathcal{A}^{(4)}_{3,D} \cup \mathcal{A}^{(4)}_{3,U})$,
and $V_4^{(0)}$ the subspace of $LG_4$ spanned by its image.
We let $\mathcal{B}_{3,D}^{(4)} = \mathcal{B}_3^{(4)} \sqcup
\mathcal{A}_{3,D}^{(4)}$,
$\mathcal{B}_{3,U}^{(4)} = \mathcal{B}_3^{(4)} \sqcup
\mathcal{A}_{3,U}^{(4)}$, and $V_4^{(D)}$, $V_4^{(U)}$ the
subspace spanned by their images, respectively.

We have $ |\mathcal{B}_3^{(4)}| = 141$ and $\dim V_4^{(0)} = 139$. Moreover,
we get $ \dim V_4^{(D)}= \dim V_4^{(U)} = 141$ hence $V_4^{(D)} = V_4^{(U)} = 
V_3+V_4$. We choose for basis of this subspace a rather arbitrary
subset of $\mathcal{B}^{(4)}_{3,D}$ of size $141$, namely $\mathcal{B}^{(4)}_{4} = \mathcal{B}^{(4)}_{3,D} \setminus
\{ s_1^{-1} s_2^{-1} s_3^{-1} s_2^{-1}, s_1 s_2^{-1} s_3^{-1} s_2^{-1} \}$,
which we check to be linearly independent.

In terms of the combinatorics of words, this enables one to pass a $s_1^{\pm}$
from over to under a pattern $s_3^{\pm}s_2^{\pm}s_3^{\pm}$, modulo terms
containing no more than one $s_3^{\pm  }$.

\begin{center}
\resizebox{!}{4cm}{\includegraphics{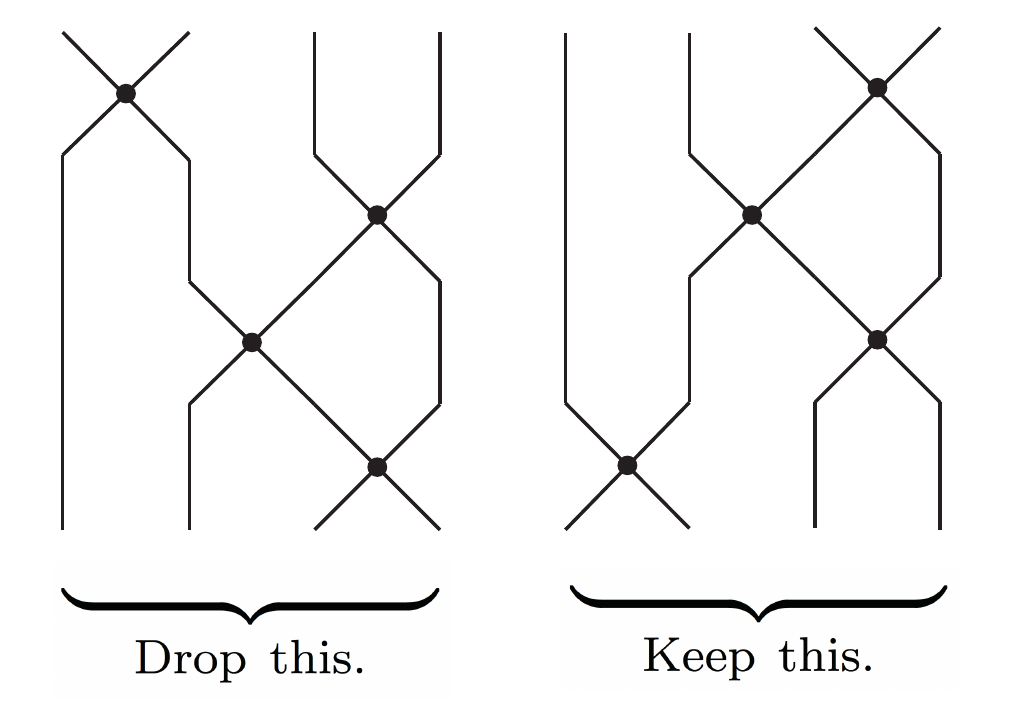}}
\end{center}

More precisely, by explicitely computing in this basis of $V_3+V_4$, we get the following
$$
\begin{array}{lcll}
s_1^{-1} (s_3^{-1} s_2 s_3^{-1}) &\equiv& (s_3^{-1} s_2 s_3^{-1}) s_1^{-1} & \mod LG_3 s_3 LG_3 + LG_3 s_3^{-1} LG_3 + LG_3 \\
s_1 (s_3^{-1} s_2 s_3^{-1}) &\equiv &(s_3^{-1} s_2 s_3^{-1}) s_1 & \mod LG_3 s_3 LG_3 + LG_3 s_3^{-1} LG_3 + LG_3
\end{array}
$$

We let $F^{\pm} \in H_4$ denote the image in $H_4$
of the expression of $s_1^{\pm} (s_3^{-1} s_2 s_3^{-1})-(s_3^{-1} s_2 s_3^{-1}) s_1^{\pm}- 
L^{\pm}$ with $L^+, L^-$ are the only linear combinations of the words of $\mathcal{B}_4^{(4)}$ such that $F^+, F^- \mapsto 0$ in $LG_4$.
We claim that $LG_4 = H_4/(r_2,F^+,F^-)$, but we actually prove more.
Indeed, we check that both these elements are non-zero exactly in all the irreducible representations of $H_4$ which
do not factorize through $LG_4$. This has for immediate consequence the following
\begin{prop} $LG_4 = H_4/(F^+) = H_4/(F^-)$.
\end{prop}

We denote $r_3$ the lifting in $K B_4$ of one of the relations $F^+$ and $F^-$, so that
$LG_4 = K B_4 / (r_1,r_2,r_3) = KB_4/(r_1,r_3)$.

\begin{verbatim}
\end{verbatim}
\begin{verbatim}
\end{verbatim}
\begin{verbatim}
\end{verbatim}

\begin{verbatim}
\end{verbatim}

\begin{figure}
\begin{center}
\resizebox{!}{4cm}{\includegraphics{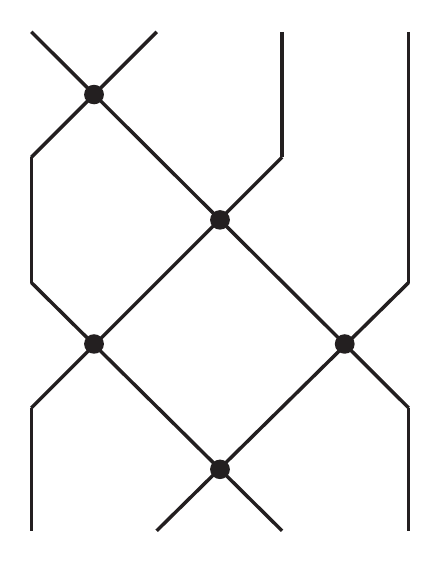}}
\resizebox{!}{4cm}{\includegraphics{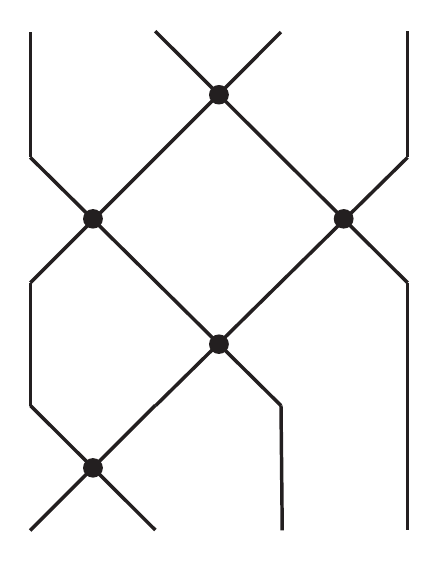}}
\end{center}
\caption{Patterns with 5 crossings}
\label{fig:M5patterns}
\end{figure}

\begin{figure}
\begin{center}
\resizebox{!}{4cm}{\includegraphics{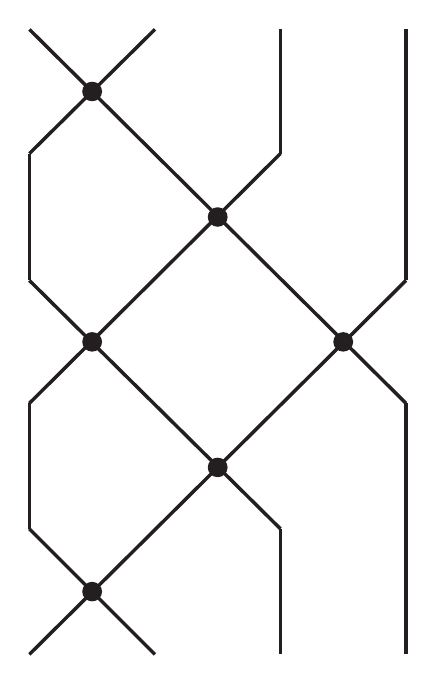}}
\end{center}
\caption{Pattern with 6 crossings}
\label{fig:M6patterns}
\end{figure}

\subsection{Case $n = 4$ : a spanning set suitable for induction.}
\label{sect6LG4bis}

We start again from the collection $\mathcal{B}_3^{(4)}$, from which
we remove the words $\{ s_1^{-1} s_2^{-1} s_3^{-1} s_2^{-1}, s_1 s_2^{-1} s_3^{-1} s_2^{-1} \}$,
and thus get a family $\mathcal{B}_{3,0}^{(4)}$  of 139 elements.
Their image in $LG_4$ is linearly independent. The words contained
in $\mathcal{B}_{3,0}^{(4)}$ all contain at most one crossing $s_3^{\pm}$
between the third and fourth strand, except for the word $s_3^{-1} s_2 s_3^{-1}$.
We want to expand this family into a basis of $LG_4$
which has the same property. For this, we add braids with 5 and 6 crossings.

We first add the 48 words corresponding to the patterns $s_1^{\alpha} s_2^{\beta} s_1^{\gamma} s_3^{\delta} s_2^{\eps}$
and $s_2^{\delta} s_3^{\eps} s_1^{\alpha} s_2^{\beta} s_1^{\gamma}$ with $\alpha, \beta, \gamma, \delta, \eps \in \{ -1,1 \}$,
which are described in figure \ref{fig:M5patterns}, with the additional property that we avoid the
subwords $s_1^{\alpha} s_2^{\beta} s_1^{\gamma} \in \{ s_1 s_2 s_1,
s_1^{-1} s_2 s_1 \}$.  This means that the subwords $s_1^{\alpha} s_2^{\beta} s_1^{\gamma} $
must belong to the image of $\mathcal{B}^{(3)}$ ; in other words,
we avoid the
patterns $s_1 s_2 s_1$ and $s_1^{-1} s_2 s_1$. This indeed provides $2 \times 2^2 \times 6 = 48$
new words.



We thus get a family $\mathcal{B}_{5}^{(4)}$
of size $139+48 = 187>175 = \dim LG_4$. The linear
span of $\mathcal{B}_5^{(4)}$, which we denote $V_5$, actually has dimension 168.
We first look for braids with 6 crossings which complete the previous family
into a spanning set of $LG_4$. We consider the pattern described 
in figure \ref{fig:M6patterns}. We are looking for a set of $7$ braids  whose
image complete the spanning of $LG_4$. A suitable set is given by
the following family

$$
\begin{array}{ll}
\mathcal{A}_6^{(4)} = \{ & s _1 s_2^{-1} s_1 s_3^{-1} s_2 s_1    ,
 s _1 s_2^{-1} s_1 s_3 s_2^{-1} s_1    ,
 s _1 s_2 s_1 s_3^{-1} s_2^{-1} s_1    ,
 s _1 s_2 s_1 s_3^{-1} s_2 s_1    ,
 s _1 s_2 s_1 s_3 s_2^{-1} s_1    , \\
&  s _1^{-1} s_2 s_1^{-1} s_3^{-1} s_2 s_1^{-1},
  s _1^{-1} s_2 s_1^{-1} s_3 s_2^{-1} s_1^{-1}     \} \\
  \end{array}
$$

\resizebox{15cm}{!}{
\begin{tikzpicture}
\braid[braid colour=red,strands=4,braid start={(0,0)}]%
{ \dsigma _1 \dsigma_2^{-1} \dsigma_1 \dsigma_3^{-1} \dsigma_2 \dsigma_1    }
\node[font=\Huge] at (4.5,-4) {\(,\)};
\braid[braid colour = red,strands=4,braid start={(4,0)}]
{ \dsigma _1 \dsigma_2^{-1} \dsigma_1 \dsigma_3 \dsigma_2^{-1} \dsigma_1    }
\node[font=\Huge] at (8.5,-4) {\(,\)};
\braid[braid colour = red,strands=4,braid start={(8,0)}]
{ \dsigma _1 \dsigma_2 \dsigma_1 \dsigma_3^{-1} \dsigma_2^{-1} \dsigma_1    }
\node[font=\Huge] at (12.5,-4) {\(,\)};
\braid[braid colour=red,strands=4,braid start={(12,0)}]%
{ \dsigma _1 \dsigma_2 \dsigma_1 \dsigma_3^{-1} \dsigma_2 \dsigma_1    }
\node[font=\Huge] at (16.5,-4) {\(,\)};
\braid[braid colour = red,strands=4,braid start={(16,0)}]
{ \dsigma _1 \dsigma_2 \dsigma_1 \dsigma_3 \dsigma_2^{-1} \dsigma_1    }
\node[font=\Huge] at (20.5,-4) {\(,\)};
\braid[braid colour = red,strands=4,braid start={(20,0)}]
{ \dsigma _1^{-1} \dsigma_2 \dsigma_1^{-1} \dsigma_3^{-1} \dsigma_2 \dsigma_1^{-1}    }
\node[font=\Huge] at (24.5,-4) {\(,\)};
\braid[braid colour=red,strands=4,braid start={(24,0)}]%
{ \dsigma _1^{-1} \dsigma_2 \dsigma_1^{-1} \dsigma_3 \dsigma_2^{-1} \dsigma_1^{-1}    }
\end{tikzpicture}}


We now select a suitable subset of the already chosen 48 words which correspond to the patterns
described in figure \ref{fig:M5patterns}. We keep all the braids corresponding to the
pattern on the right-hand side of figure \ref{fig:M5patterns} \emph{but} the words
$$\mathcal{D}^{(4)}_{5,R} = \{ s_2s_3s_1^{-1}s_2^{-1} s_1^{-1},
s_2s_3s_1^{-1}s_2^{-1} s_1,
s_2s_3s_1 s_2^{-1} s_1^{-1}, 
s_2s_3s_1 s_2 s_1^{-1} \}. $$
We denote $\mathcal{A}^{(4)}_{5,R}$ the corresponding set of $24-4 = 20$ words.
Finally, among the $24$ words corresponding to the pattern on the left-hand side
of figure \ref{fig:M5patterns}, we keep only the 9 words
$$
\begin{array}{ll}
\mathcal{A}_{5,L}^{(4)} = & \{ s_1^{-1}s_2^{-1}s_1^{-1}s_3s_2, s_1^{-1}s_2^{-1}s_1s_3s_2, s_1^{-1}s_2s_1^{-1}s_3s_2, s_1s_2^{-1}s_1^{-1}s_3s_2, s_1s_2^{-1}s_1s_3s_2, \\ &
  s_1s_2s_1^{-1}s_3s_2, s_1^{-1}s_2^{-1}s_1^{-1}s_3s_2^{-1}, s_1^{-1}s_2s_1^{-1}s_3s_2^{-1}, s_1s_2^{-1}s_1^{-1}s_3s_2^{-1} \}. \\
  \end{array}
$$
The final collection $\mathcal{B}_6^{(4)} = \mathcal{B}_{3,0}^{(4)} \sqcup \mathcal{A}_{5,L}^{(4)} \sqcup \mathcal{A}_{5,R}^{(4)}
\sqcup \mathcal{A}_{6}^{(4)}$ has cardinality 175. By computer calculation one can then check the following.


\begin{prop} The collection $\mathcal{B}_6^{(4)}$ is a basis for $LG_4$. The collection of the $b_1 s_3^r b_1$ for $b_1,b_2
\in \mathcal{B}^{(3)}$ and $r \in \{-1,0,1 \}$ 
spans in $LG_4$ a subset of dimension $174 = \dim LG_4 - 1$. 
In particular, we have
 $$
 LG_4 =  \left( \sum_{r \in \{-1,0,1 \}}  LG_3 s_3^{r} LG_3 \right) \oplus K s_3^{-1} s_2 s_3^{-1}
 $$ 
\end{prop}

\begin{verbatim}
\end{verbatim}
\begin{verbatim}
\end{verbatim}
\subsection{$LG_{n+1}$ as $LG_n$-bimodule}
\label{sect6bimod}

\begin{lemma} {\ } \label{lemreduction}

\begin{enumerate}
\item $LG_{n-1} s_{n-1}^{-1} s_{n-2} s_{n-1}^{-1} LG_{n-1} \subset \sum_r LG_{n-1} s_{n-1}^r LG_{n-1}
+ LG_{n-3} s_{n-1}^{-1} s_{n-2} s_{n-1}^{-1}$
\item $LG_{n-1} s_{n-1}^{-1} s_{n-2} s_{n-1}^{-1}LG_{n-1} s_{n-1}^{\pm 1} \subset \sum_r LG_{n-1} s_{n-1}^r LG_{n-1}
s_{n-1}^{\pm 1}+ LG_{n-3} s_{n-1}^{-1} s_{n-2} s_{n-1}^{-1}$
\end{enumerate}
\end{lemma}
\begin{proof}
Part (ii) is a straightforward consequence of (i) and of the choice of basis for $LG_3$. We prove part (i). When $n = 4$
this is a consequence of our choice of basis.
Let $A = \sum_r LG_{n-1} s_{n-1}^r LG_{n-1}
+ LG_{n-3} s_{n-1}^{-1} s_{n-2} s_{n-1}^{-1} \subset LG_n$.
Since $ s_{n-1}^{-1} s_{n-2} s_{n-1}^{-1} \in A$, we only need to prove that $A$
is a $LG_{n-1}$-submodule on both sides. Because $LG_{n-3}$ commutes with $s_{n-1}^{-1} s_{n-2} s_{n-1}^{-1}$
this amounts to saying that $A$ is stable by multiplication on both sides by $s_{n-2}^{\pm 1}$ and $s_{n-3}^{\pm 1}$,
that is $s_{n-2}^{\pm 1}(s_{n-1}^{-1} s_{n-2} s_{n-1}^{-1}) \in A$, $s_{n-3}^{\pm 1}(s_{n-1}^{-1} s_{n-2} s_{n-1}^{-1}) \in A$,
$(s_{n-1}^{-1} s_{n-2} s_{n-1}^{-1})s_{n-2}^{\pm 1} \in A$, $(s_{n-1}^{-1} s_{n-2} s_{n-1}^{-1})s_{n-3}^{\pm 1} \in A$.
These eight braids are conjugates of braids whose image lie in $LG_4$ and we can use the result for $n = 4$ and $n=3$, that we obtained above, in order to
conclude the proof.

$$
\begin{array}{|c|c|c|c|}
\hline
\resizebox{!}{3cm}{\mbox{\begin{tikzpicture}
\braid[braid colour=blue,strands=4,braid start={(0,0)}]%
{ \dsigma _2}
\braid[braid colour=red,strands=4,braid start={(0,-1)}]%
{ \dsigma _3^{-1}  \dsigma_2  \dsigma_3 ^{-1}}
\end{tikzpicture} }}& 
\resizebox{!}{3cm}{\mbox{\begin{tikzpicture}
\braid[braid colour=blue,strands=4,braid start={(0,0)}]%
{ \dsigma _2^{-1}   }
\braid[braid colour=red,strands=4,braid start={(0,-1)}]%
{  \dsigma _3^{-1}  \dsigma_2  \dsigma_3 ^{-1}}
\end{tikzpicture}}} &
\resizebox{!}{3cm}{\mbox{\begin{tikzpicture}
\braid[braid colour=blue,strands=4,braid start={(0,0)}]%
{ \dsigma _1}
\braid[braid colour=red,strands=4,braid start={(0,-1)}]%
{ \dsigma _3^{-1}  \dsigma_2  \dsigma_3 ^{-1}}
\end{tikzpicture} }}& 
\resizebox{!}{3cm}{\mbox{\begin{tikzpicture}
\braid[braid colour=blue,strands=4,braid start={(0,0)}]%
{ \dsigma _1^{-1}   }
\braid[braid colour=red,strands=4,braid start={(0,-1)}]%
{  \dsigma _3^{-1}  \dsigma_2  \dsigma_3 ^{-1}}
\end{tikzpicture}}} \\
\hline
\resizebox{!}{3cm}{\mbox{\begin{tikzpicture}
\braid[braid colour=red,strands=4,braid start={(0,0)}]%
{ \dsigma _3^{-1}  \dsigma_2  \dsigma_3 ^{-1}}
\braid[braid colour=blue,strands=4,braid start={(0,-3)}]%
{ \dsigma _2}
\end{tikzpicture} }}& 
\resizebox{!}{3cm}{\mbox{\begin{tikzpicture}
\braid[braid colour=red,strands=4,braid start={(0,0)}]%
{  \dsigma _3^{-1}  \dsigma_2  \dsigma_3 ^{-1}}
\braid[braid colour=blue,strands=4,braid start={(0,-3)}]%
{ \dsigma _2^{-1}   }
\end{tikzpicture}}} &
\resizebox{!}{3cm}{\mbox{\begin{tikzpicture}
\braid[braid colour=red,strands=4,braid start={(0,0)}]%
{ \dsigma _3^{-1}  \dsigma_2  \dsigma_3 ^{-1}}
\braid[braid colour=blue,strands=4,braid start={(0,-3)}]%
{ \dsigma _1}
\end{tikzpicture} }}& 
\resizebox{!}{3cm}{\mbox{\begin{tikzpicture}
\braid[braid colour=red,strands=4,braid start={(0,0)}]%
{  \dsigma _3^{-1}  \dsigma_2  \dsigma_3 ^{-1}}
\braid[braid colour=blue,strands=4,braid start={(0,-3)}]%
{ \dsigma _1^{-1}   }
\end{tikzpicture}}} \\
\hline
\end{array}
$$

\end{proof}

\begin{theor} {\ } \label{theobimoduleLG} For $n \geq 3$ we have
\begin{enumerate}
\item $
LG_n = LG_{n-1} s_{n-1}^{\pm 1} LG_{n-1} + \sum_{k+ \ell = n} LG_k LG_l
$
\item $
LG_n =   \sum_r  LG_{n-1} s_{n-1}^{r} LG_{n-1}+ LG_{n-3}( s_{n-1}^{-1} s_{n-2} s_{n-1}^{-1})
$
\end{enumerate}
\end{theor} 
\begin{proof}
(second formula). By induction on $n$, the cases $n=3$ and $n = 4$ being already done.
Let $A = \sum_r  LG_{n-1} s_{n-1}^{r} LG_{n-1}+ LG_{n-3}( s_{n-1}^{-1} s_{n-2} s_{n-1}^{-1}) \subset LG_n$.
Since $1 \in A$, we only need to prove that $s_{k}^{\pm 1}A \subset A$ for $k \in \{ n-1,n-2,n-3 \}$. For $k < n-1$
this is a straightforward consequence of lemma \ref{lemreduction} (i), so we can assume $k = n-1$.
We have, for some $u$, $s_{n-1}^{\pm 1} LG_{n-3}( s_{n-1}^{-1} s_{n-2} s_{n-1}^{-1}) = LG_{n-3}( s_{n-1}^{u} s_{n-2} s_{n-1}^{-1})\subset LG_{n-3}s_{n-1}^{-1} s_{n-2} s_{n-1}^{-1} + \sum_r LG_{n-1} s_{n-1}^r LG_{n-1} \subset A$
because of the chosen basis for $LG_3$. We then only need to prove $s_{n-1}^{\pm}  LG_{n-1} s_{n-1}^r LG_{n-1} \subset A$.
We use the induction assumption $LG_{n-1} \subset \sum_{r'} LG_{n-2} s_{n-2}^{r'}LG_{n-2} + 
  LG_{n-4}( s_{n-2}^{-1} s_{n-3} s_{n-2}^{-1})$ hence
$$
\begin{array}{lcl}
s_{n-1}^{\pm 1} LG_{n-1} s_{n-1}^r LG_{n-1} &\subset &
\sum_{r'} s_{n-1}^{\pm 1}LG_{n-2} s_{n-2}^{r'}LG_{n-2}s_{n-1}^r LG_{n-1} \\
& & + 
  s_{n-1}^{\pm 1}LG_{n-4}( s_{n-2}^{-1} s_{n-3} s_{n-2}^{-1})s_{n-1}^r LG_{n-1}\\ 
 &\subset &
\sum_{r'} LG_{n-2}  s_{n-1}^{\pm 1}s_{n-2}^{r'}s_{n-1}^rLG_{n-2} LG_{n-1} \\
& & + 
  LG_{n-4} s_{n-1}^{\pm 1}( s_{n-2}^{-1} s_{n-3} s_{n-2}^{-1})s_{n-1}^r LG_{n-1}\\ 
\end{array}
$$
and, by the case $n = 3$,
$s_{n-1}^{\pm 1}s_{n-2}^{r'}s_{n-1}^r \in \sum_u LG_{n-1} s_{n-1}^u LG_{n-1} + LG_{n-1}  s_{n-1}^{-1} s_{n-2} s_{n-1}^{-1}LG_{n-1}$,
hence $LG_{n-1} s_{n-1}^{\pm 1}s_{n-2}^{r'}s_{n-1}^r LG_{n-1}$.
 Moreover, $s_{n-1}^{\pm 1}( s_{n-2}^{-1} s_{n-3} s_{n-2}^{-1})s_{n-1}^r
 \in \sum_u LG_{n-1} s_{n-1}^u LG_{n-1} + LG_{n-1} s_{n-1}^{-1} s_{n-2} s_{n-1}^{-1} LG_{n-1}$ because of
the chosen basis in $LG_4$, hence
$$
s_{n-1}^{\pm 1} LG_{n-1} s_{n-1}^r LG_{n-1} \subset 
\sum_u LG_{n-1} s_{n-1}^u LG_{n-1} + LG_{n-1} s_{n-1}^{-1} s_{n-2} s_{n-1}^{-1} LG_{n-1} \subset A
$$
by lemma \ref{lemreduction}.
\end{proof}

\begin{verbatim}
\end{verbatim}
\begin{verbatim}
\end{verbatim}

\section{Markov traces}
\label{Markovtrace}
Using the careful analysis of the previous section, we define a quotient $A_n$ of the braid group algebra $KB_n$ of $n$ strands by a cubical relation $r_1$  as well as one relation $r_2$ on three strands and one relation $r_3$ on four strands. Notice that we conjecture (see conjecture \ref{conjiso}) that this algebra is isomorphic for all $n$ to the centralizer algebra $LG_n$.

Since $A_4 \simeq LG_4$, the proof of theorem \ref{theobimoduleLG}
can be adapted immediately to yield the following statement.
\begin{theor}
\label{theobimoduleA}
For all $n \geq 3$,
$$
A_{n+1} = A_n + A_n s_n A_n + A_n s_n^{-1} A_n + A_{n-2} s_{n-1}^{-1} s_n s_{n-1}^{-1}.
$$
\end{theor}

This implies immediately that $A_n$ is finite dimensional for all $n$.
The precise dimension of $A_n$ is the content of conjecture \ref{conjdimAn}.\\

Using the methods of section \ref{soussectchar}, the character
table of $\Gamma_5$ and the fact (see \cite{G32}) that $H_5$ is a flat
deformation of $K \Gamma_5$ enable us to get
the list of irreducible representations of $H_5$ which
factor through $H_5$. This yields $\dim A_5 = 1764 = \dim LG_5$, whence
the following evidence for conjecture \ref{conjiso}.

\begin{theor} $A_5 \simeq LG_5$.
\end{theor}

The main result of this section is that the tower of algebras $(A_n)_{n \geq 1}$ can be endowed  with a unique trace $Tr_n$ which computes the Links-Gould invariant. In addition
we prove that the relations $r_1$, $r_2$ and $r_3$ are a complete set of relations for the Links-Gould invariant, i.e. one can recursively compute the Links-Gould invariant using this relations. 

Given an integer $n\geq 1$, consider the natural embedding of $B_{n}$ into $B_{n+1}$. Denote by $\phi_n$ its extension to an homomorphism from $A_n$ to $A_{n+1}$. 

\begin{theor} \label{theounicite}
For $z \in K$, there exists a family of traces $Tr_n:A_n\rightarrow K$, $n \geq 1$,
such that
\begin{itemize}
\item $Tr_{n+1}(\phi_n(\beta))=zTr_n(\beta)$ for all $\beta \in A_n$.
\item $Tr_n(\alpha \beta)=Tr_n(\beta \alpha)$ for all $\alpha$, $\beta\in A_n$.
\item $Tr_{n+1}(\phi_n(\beta) s_n^{\pm 1}) = Tr_n(\beta)$ for all $\beta \in A_n$ (`Markov property').
\item $Tr_1=1$
\end{itemize}
if and only if $z = 0$. If $z=0$, this family is unique. The same statements hold with $(A_n)$ replaced by $(LG_n)$.
\label{trace}
\end{theor}

\begin{proof}
First we prove that if the trace exists then $z$ is equal to zero. Since $A_4$ is semi-simple, it is isomorphic to a direct sum of ten matrix algebras and therefore a trace on $A_4$ is a linear combination of matrix traces. By using the Markov properties above,
we get that the value of $Tr_4$ on the family $\mathcal{F} = (s_3,
s_1 s_3, s_1^{-1} s_3, s_2 s_3, s_2^{-1} s_3, s_1 s_2 s_3, s_1 s_2^{-1} s_3,
s_1^{-1} s_2 s_3, s_1^{-1} s_2^{-1} s_3,s_3^{-1} )$ is $(z^2,z,z,z,z,1,1,1,1,z^2)$.
On the other hand, we check by computer that the values of the 10 matrix traces
on this family provide an invertible $10 \times 10$ matrix. As a consequence,
the values of $Tr_4$ on this family determines its value on arbitrary
elements of $B_4$, as polynomial functions of $z$. This enables us to compute
the value of $Tr_4$ on $s_1^{-1} s_3^{-1}$ and $s_2^{-1} s_3^{-1}$. Since
its value has to be $z$ in both cases by the Markov property, we get two equations
on $z$, which have the form $\alpha' z (z-\alpha) = 0$ and
$\beta' z (z-\beta) = 0$ for some $\alpha, \alpha', \beta, \beta' \in K^{\times}$,
with $\alpha \neq \beta$. This clearly implies $z = 0$.

Existence follows from the existence of the Links-Gould invariant and unicity from the careful analysis of the previous section.
In more details, define $T_n$ by $T_n(\beta)=\mathfrak{LG}(\widehat{\beta})$ for all $\beta$ and $n\geq 1$. First $T_n$ is well defined on $A_n$ since the relations $r_1$, $r_2$ and $r_3$ are satisfied by the Links-Gould invariant. In addition since the Links-Gould invariant vanishes on split links, it implies that $T_n(\beta)=0$ for all $\beta\in \mathrm{Im}(\phi_{n-1})$ and $T_1=1$ is by the normalization of the Links-Gould invariant (one on the unknot).
We have also $T_{n+1}(\beta s_n^{\pm 1})=T_n(\beta)$ for all $\beta \in \mbox{Im}(\phi_n)$,
since the Links-Gould invariant is invariant under the first Reidemeister move. It remains to say that $T_n(\beta \alpha)=T_n(\alpha \beta)$ is satisfied  because the Links-Gould invariant is trully an invariant of links, it does in particular not depend on where you open the link to compute it (see Remark 5.1 in \cite{PATGEER}). Given two braids $\alpha$ and $\beta$ consider the topological partial closure of the braids $\alpha\beta$ and $\beta\alpha$ decribed in Figure \ref{partial}. It can be easily seen that these are two different openings of the topological closure of $\alpha\beta$ (which is of course isotopic to the topological closure of $\beta\alpha$.) This finishes the proof of the existence of $Tr_n$. \\
\begin{center}
\begin{figure}[h!]
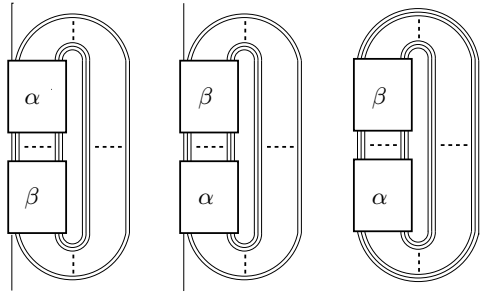
\caption{ Partial closure of $\alpha\beta$ and $\beta\alpha$ and closure of $\alpha\beta$.}
\label{partial}
\end{figure}
\end{center}
The algebra $A_3$ and $A_4$ are respectively isomorphic to the algebras $LG_3$ and $LG_4$. 


We use the bimodule decomposition afforded by theorem \ref{theobimoduleA}.
Suppose by induction that $Tr_k$ is unique for $k\leq n$. Let  $\gamma$  be an element of in $A_{n+1}=A_n s_n^r A_n+ A_{n-2}s_n^{-1}s_{n-1}s_n^{-1}$ ($r\in\{-1,0,+1\}$) we can suppose that either $\gamma=\alpha_1 s_n^r \alpha_2$ or $\gamma= \beta s_n^{-1}s_{n-1}s_n^{-1}$ with $\alpha_1$, $\alpha_2$ $\in A_n$ and $\beta \in A_{n-2}$.
In the first case if $r=0$ we have $Tr_{n+1}(\gamma)=0$ and if $r=\pm 1$ we have $Tr_{n+1}(\gamma)=Tr_n(\alpha_1 \alpha_2)$. In the second case we have $Tr_{n+1}(\gamma)=Tr_{n+1}(\beta s_n^{-2}s_{n-1})
=Tr_{n+1}(s_{n-1} \beta s_n^{-2})
$. By applying the cubic relation $r_1$ to the factor $s_n^{-2}$ we reduce to the previous case. Hence $Tr_{n+1}$ is unique. 
It is direct computation given the basis for $A_1$, $A_2$ and $A_3$ to prove the trace is unique for these algebras. This finishes the proof of unicity.

The case of $LG_n$ is similar, since $LG_4 = A_4$.
\end{proof}

\begin{remark}
Given an integer $n\geq 1$, for all $1\leq k\leq n$ consider the natural embedding of $B_k\times B_{n-k}$ into $B_n$ (see Figure (\ref{inj})). Denote by $\phi_k$ its extension to an homomorphism from $A_k\otimes A_{n-k}$ to $A_n$. Define $I_n$ the subvector space of $A_n$ generated by the images of the $\phi_k$ ($1\leq k \leq n$). 

\begin{center}
\begin{figure}[h!]
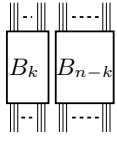
\caption{Injection of $B_k\times B_{n-k}$ into $B_n$.}
\label{inj}
\end{figure}
\end{center}

The fact that $z$ is equal to zero and an induction argument shows that the unique trace $Tr_n$ on $A_n$ vanishes on $I_n$. This implies that the Links-Gould invariant vanishes on split links (see \cite{Ishiisplit} for a different proof).
\end{remark}

\begin{cor}
The relations $r_1$, $r_2$ and $r_3$ are a complete set of skein relations for the Links-Gould invariant.
\end{cor}
\begin{proof}
It is an immediate consequence of Theorem \ref{trace}.
\end{proof}

\begin{remark}
 Notice first that the relations $r_1$, $r_2$ and $r_3$ are sufficient to compute the Links-Gould invariant. In addition using the representations of section \ref{secthecke}, one can deduce relations expressing the elements $s_3^{\pm}s_2^{-1}s_1s_2^{-1}s_3^{\pm}$, 
 in the chosen basis of $A_4$
 which could in pratice simplify a recursive computation. All these relations are of course consequences of $r_1$, $ r_2$ and $r_3$
 \end{remark}

As noticed in the proof of the theorem, the Markov trace on $A_4$ is a linear
combination of the matrix traces on the irreducible
representations of $A_4$, namely
$$
Tr_4 = \sum_{i=1}^{10} a_{\chi_i} \mathrm{tr}_{\chi_i}
$$
where $i \in \{ 1,\dots, 10 \}$,
$a_{\chi_i} \in K$, and
$\chi_i$ is the $i$-th irreducible representation of $A_4$, following the order
chosen at the end of section \ref{secthecke}. For comparaison with other
traces factoring through $H_4$, we provide these coefficients
in table \ref{tablecoefs}.
\begin{table}
$$
a_{\chi_1} = {\frac {{c}^{2}{a}^{4}{b}^{2}+{a}^{4}bc+{a}^{4}-{c}^{2}{a}^{3}{b}^{3}-{a}^{3}{c}^{3}{b}^{2}+{b}^{3}{c}^{3}{a}^{2}-{a}^{2}{c}^{2}
{b}^{2}+a{c}^{2}{b}^{3}+a{c}^{3}{b}^{2}+{b}^{2}{c}^{2}}{a \left( a-b \right)  \left( a-c \right)  \left( {a}^{2}+{c}^{2}
 \right)  \left( {b}^{2}+{a}^{2} \right) }}
$$
{}
$$
a_{\chi_2} =-{\frac { \left( ab+1 \right)  \left( {a}^{3}b-c{a}^{2}b-ab+{c}^{2} \right) {c}^{2}}{ \left( a-c \right)  \left( b-c \right) 
 \left( {a}^{2}+{c}^{2} \right)  \left( b{a}^{2}-{c}^{3} \right) }}
, \ \ \ 
a_{\chi_3} ={\frac { \left( ac+1 \right)  \left( {a}^{3}c-c{a}^{2}b-ac+{b}^{2} \right) {b}^{2}}{ \left( a-b \right)  \left( b-c \right) 
 \left( {b}^{2}+{a}^{2} \right)  \left( -{b}^{3}+c{a}^{2} \right) }}
$$
{}
$$
a_{\chi_4} =-{\frac { \left( ac+1 \right)  \left( c{a}^{2}-cab-a-c \right) {b}^{2}}{a \left( a-b \right)  \left( b-c \right)  \left( b+c
 \right)  \left( {b}^{2}+{a}^{2} \right) }}
$$
{}
$$
a_{\chi_5} =-{\frac { \left( ab+1 \right)  \left( {a}^{4}bc-b{c}^{2}{a}^{3}+{a}^{3}b-{a}^{3}c-{b}^{3}{a}^{2}c+{a}^{2}{c}^{2}{b}^{2}-c{a}^{2}
b+a{c}^{2}{b}^{3}-a{c}^{3}{b}^{2}+{b}^{2}ac+c{b}^{3}-{b}^{2}{c}^{2} \right) c}{a \left( a-c \right)  \left( b-c \right)  \left( 
{a}^{2}+{c}^{2} \right)  \left( -{b}^{3}+c{a}^{2} \right) }}
$$
{}
$$
a_{\chi_6} = {\frac { \left( ab+1 \right)  \left( b{a}^{2}-cab-a-b \right) {c}^{2}}{a \left( a-c \right)  \left( b-c \right)  \left( b+c
 \right)  \left( {a}^{2}+{c}^{2} \right) }}
$$
{}
$$
a_{\chi_7} ={\frac { \left( ac+1 \right)  \left( {a}^{4}bc-c{a}^{3}{b}^{2}+{a}^{3}c-{a}^{3}b-{a}^{2}{c}^{3}b+{a}^{2}{c}^{2}{b}^{2}-c{a}^{2}b
+a{c}^{3}{b}^{2}+a{c}^{2}b-a{c}^{2}{b}^{3}-{b}^{2}{c}^{2}+{c}^{3}b \right) b}{a \left( a-b \right)  \left( b-c \right)  \left( {
b}^{2}+{a}^{2} \right)  \left( b{a}^{2}-{c}^{3} \right) }}
$$
{}
$$
a_{\chi_8} ={\frac {{c}^{2}b \left( ac+1 \right)  \left( ab+1 \right) }{a \left( b-c \right)  \left( b+c \right)  \left( b{a}^{2}-{c}^{3}
 \right) }}
, \ \ 
a_{\chi_9} =-{\frac {{b}^{2}c \left( ac+1 \right)  \left( ab+1 \right) }{a \left( b-c \right)  \left( b+c \right)  \left( -{b}^{3}+c{a}^{2}
 \right) }}
$$
{}
$$
a_{\chi_{10}} ={\frac {cb \left( ac+1 \right)  \left( ab+1 \right)  \left( bc+{a}^{2} \right) }{a \left( b{a}^{2}-{c}^{3} \right)  \left( -{b}^
{3}+c{a}^{2} \right) }}
$$

\caption{Coefficients of the unique Markov trace on $A_4 = LG_4$.}
\label{tablecoefs}
\end{table}

\section{On the image and kernel of $B_n \to LG_n(\alpha)^{\times}$}
\label{sectimagenoyau}

We recall that $\mathcal{LG}_n(\alpha)$ has been defined
in section \ref{sectstruct} as the \emph{Lie} subalgebra of $\End(V(0,\alpha)^{\otimes n})$ generated
by the $\Omega_{ij}'s$.

\subsection{Structure of  $\mathcal{LG}_n(\alpha)$ and Zariski closure of $B_n$ inside
$LG_n(\alpha)^{\times}$}

The proof of the following proposition is parallel to its analogue for $BMW$
(see \cite{BMW} prop. 5.1),
the algebra $LG_n(\alpha)$ playing the role of the Birman-Wenzl-Murakami algebra
thanks to theorem \ref{theosurjcommut}.

\begin{prop} For generic values of $\alpha$, the Lie algebra $\mathcal{LG}_n(\alpha)$
is reductive with center $\kk T_n$.
\end{prop}

For $\la = [a,k]_r$ we let $\rho_{\la} : \mathcal{LG}_n(\alpha) \to \gl_{d(\la)}(\kk)$
denote the corresponding representation of $\mathcal{LG}_n(\alpha)$,
and $\rho'_{\la}$ its restriction to the derived Lie algebra $\mathcal{LG}'_n(\alpha)$.
The proof of the following proposition is parallel
to propositions 5.5 and 5.6 of \cite{BMW}.
\begin{prop} For generic $\alpha$, $\rho'_{\la_1}$ is isomorphic to $\rho'_{\la_2}$
if and only if $\la_1 = \la_2$ or $d(\la_1) = d(\la_2) = 1$. If $d(\la_1),d(\la_2)>1$,
then the dual representation of $\rho_{\la_2}$ cannot be isomorphic to $\rho_{\la_1}$. 
\end{prop}
\begin{proof}
As in \cite{BMW}, $\mathcal{LG}_n'$ is generated by the $t'_{ij} = t_{ij}-2T_n/n(n-1)$
with $T_n = \sum_{1 \leq r,s \leq n} t_{rs}$ and $T_n$ acts by a scalar
on a given irreducible representation. For $\rho'_{\la_1}$ and
$\rho'_{\la_2}$ to be isomorphic, $\rho'_{\la_1}(t'_{12})$ and
$\rho'_{\la_2}(t'_{12})$ should be conjugate, hence $(\rho_{\la_1}(T_n) - \rho_{\la_2}(T_n))/n(n-1)$
should then belong to $X - X$, where $X =\{ -2 \alpha^2, -2 (\alpha+1)^2, -2\alpha(\alpha+1) \}$
is the set of possible eigenvalues for $\rho_{\la_1}(t_{ij})$ and $\rho_{\la_2}(t_{ij})$.
For generic $\alpha$ the set $X +(-X) = \{ 0 \} \sqcup (-2)\times \{ 2 \alpha+1, \alpha + 1, \alpha, - \alpha, -\alpha-1,-2 \alpha -1 \}$
has cardinality $7$. This implies that either
$(\rho_{\la_1}(T_n) - \rho_{\la_2}(T_n))/n(n-1) = 0$, or both $\rho_{\la_1}(t_{12})$ and
$\rho_{\la_2}(t_{12})$ have a single eigenvalue (as,  for $x,y \in X$, $x-y$ then determines $x$ and $y$
when $x -y\neq 0$). By semisimplicity of the action of $t_{12}$ this implies that the images of
the $t_{ij}$'s are scalars in both $\rho_{\la_1}$ and $\rho_{\la_2}$. By irreducibility this implies
that $\dim \rho_{\la_1} = \dim \rho_{\la_2} = 1$.
The case of the dual representations is similar, using that, for $x \in X$, $2x \in X_+^{=} =  (-2) \times \{ 2 \alpha^2,
2 \alpha^2 + 2 \alpha, 2 \alpha^2 + 4 \alpha + 1 \}$ and,
for $x,y \in X$ with $x \neq y$, $x+y \in X_+^{\neq} = (-2) \times \{ 2 \alpha^2 + 2 \alpha +1, 2 \alpha^2 + \alpha , 2 \alpha^2 + 3 \alpha +1 \}$ ;
again, for generic $\alpha$, $X_+^{\neq} \cap X_+^{=} = \emptyset$, $X_+^{\neq}=3$,
and we get a contradiction unless $\dim \rho_{\la_1} = \dim \rho_{\la_2} = 1$.

\end{proof}

We let $V_{[a,k]_r}$ denote the $[a,k]_r$-component of $LG_n$.

\begin{theor} \label{theoimlie} For generic $\alpha$, the image of $\mathcal{LG}'_n(\alpha)$ inside $LG_n(\alpha) = \bigoplus \End(V_{[a,k]_r})$
is $\bigoplus \sl(V_{[a,k]_r})$.
\end{theor}
\begin{proof}
The proof is by induction on $r$, and follows the same general pattern as in \cite{BMW}.
We will use freely the Lie-theoretic results of \cite{IH2}. Assuming the result known for
$r-1$ (and the cases $r \leq 2$ being trivial), we only need to check that, for $\la = [a,k]_r$,
$\rho_{\la}(\mathcal{LG}_n) = \sl(V_{[a,k]_r})$, by the same argument as in \cite{BMW,LIETRANSP}.
By abuse, we denote $\la = \rho_{\la}$, $\Res_i$ the restriction from $\mathcal{LG}_r$
to $\mathcal{LG}_i$ for $i \leq r$, and $\Res = \Res_{r-1}$.
First note that, if $\Res \la$ has $s$ irreducible components $\mu_1,\dots,\mu_s$
occuring with multiplicity one, then by the induction assumption the rank
of $\rho_{\la}(\mathcal{LG}_{r-1})$ is $(\sum \dim \mu_j) - s = \dim V_{\la} - s$,
hence the rank of $\rho_{\la}(\mathcal{LG}_{r-1})$ is at least $\dim V_{\la} - s > (\dim V_{\la})/2$
as soon as $\dim V_{\la} > 2s$. By \cite{IH2} lemma 3.1 this implies the conclusion
$\rho_{\la}(\mathcal{LG}_{r-1})$.

We will show that this assumption is almost always satisfied, and we will check
separately the remaining cases. First note that the restriction is always multiplicity
free, and that the number of components is at most 4. More precisely, for $a \leq r-1$
and $k \leq r-a-1$, $\Res [a,k]_r$ contains
\begin{enumerate}
\item $[a+1,k-1]_{r-1}$ if $a \leq r-3$ and $1 \leq k \leq r-a-2$
\item $[a,k]_{r-1}$ if $a \leq r-2$ and $k \leq r-a-2$
\item $[a,k-1]_{r-1}$ if $a \leq r-2$ and $1 \leq k$
\item $[a-1,k]_{r-1}$ if $a \geq 1$
\end{enumerate}

Moreover we notice that the assumptions of lemma 3.3 (I) of \cite{IH2}
are satisfied as soon as $\dim \la < (1+ \rk \h)^2$ : under this
condition, this lemma implies that $\g$ is simple.

We first deal with a few special cases. For a given $\la$,
we denote $\g$ the image of $\mathcal{LG}'_r$ and $\h$ the
image of $\mathcal{LG}'_{r-1}$.

\medskip

\noindent {\it  The cases $[a,0]_r$, $[r-2,1]_r$ and $[1,r-2]_r$.} \\
The 1-dimensional cases $[0,0]_r$ and $[r-1,0]_r$ are trivial.
We have $\Res [1,0]_r = [0,0]_{r-1} + [1,0]_{r-1}$, which implies $\dim [1,0]_r = r-1$.
Let $\la = [1,0]_3$. Since $\dim \la = 2$, $\g \subset \sl_2$ and $\g$ is semisimple hence $\g = \sl_2$.
Let $\la = [1,0]_4$. Then $\h = \sl_2 \subset \g \subset \sl_3$. If $\g$ has rank $2>3/2$ then
$\g = \sl_3$ by \cite{IH2} lemma 3.1 ; otherwise $\g$ has rank 1, that is $\g \simeq \sl_2$. But
the only 3-dimensional irreducible representation of $\sl_2$ is selfdual, which excludes this case.
Let $\la = [1,0]_5$. Then $\h = \sl_3 \subset \g \subset \sl_4$. If $\rk \g > 2$ then $\g = \sl_4$
by \cite{IH2} lemma 3.1 ; otherwise $\g$ has rank 2; since $4 < (2+1)^2$, by lemma 3.3 of \cite{IH2}
$\g$ is simple. Since the $4$-dimensional representations of $\so_5$ and $\sp_4$ are selfdual,
this implies $\g = \sl_4$. Let now $\la = [1,0]_r$ for $r > 5$. Then $\rk \h = r-3 > (r-1)/2 = (\dim \la)/2$
and $\g = \sl_{r-1}$, which concludes the case of $[1,0]_r$. The case of $[r-2,0]_r$ is similar.

Let now $\la = [a,0]_r$ for $a > 1$ and $a < r-2$, which implies $r \geq 5$.
Then $\Res \la = [a-1,0]_{r-1} + [a,0]_{r-1}$ hence we need to show that
$\dim \la > 4$ in order to apply lemma 3.1 of \cite{IH2}. It is easily checked
that $\Res_4 [a,0]_r$ always contains $[1,0]_4 + [2,0]_4$, of dimension $3+3 = 6$,
and this concludes the case of the $[a,0]_r$.

The same arguments provide the case of the $[r-2,1]_r$ and of the $[1,r-2]_r$.

\medskip

\noindent {\it  The cases of $[0,1]_3$, $[0,1]_4$, $[0,2]_4$.} \\
Let $\la = [0,1]_3$, of dimension $3$. Then $\g \subset \sl_3$ ;
if $\rk \g = 1$ then $\g \simeq \sl_2 \simeq \so_3$, but the 3-dimensional irreducible representation
of $\g$ is selfdual, a contradiction. Thus $\rk \g = 2 > 3/2$ and $\g \simeq \sl_3$ by lemma 3.1 of \cite{IH2}.

Let $\la = [0,1]_4$, of dimension 6. We have $\Res \la = [0,0]_3 + [0,1]_3 + [1,0]_3$ hence
$\h \simeq \sl_3 \times \sl_2$ and $\rk \g \geq 3$. If $\rk \g > 3  = 6/2$ we are done, so we need
to exclude the case $\rk \g = 3$. In that case, by \cite{IH2} lemma 3.3, we get that $\g$ is simple,
and by \cite{IH2} lemma 3.4 $\g \simeq \sl_4$ and $\la$ is selfdual, a contradiction.
The case of $\la = [0,2]_4$ is similar to $[0,1]_4$.

\medskip

\noindent {\it  The case of $[a,r-a-1]_r$ for $1 \leq a \leq r-3$.} \\
This case implies $r \geq 4$. Let $\la = [a,r-a-1]_r$. Then
$\Res \la = [a,r-a-2]_{r-1} + [a-1,r-a-1]_{r-1}$, and we have the conclusion
as soon as $\dim \la > 4$. Note that if $a \leq r-4$ (resp. $a \geq 2$) then
$[a,r-a-2]_{r-1}$ belongs to the case under study for $r-1$ : by induction
this shows that $\dim \la > 4$ for $r \geq 6$, as $[1,4]_6$, $[2,3]_6$, $[3,2]_6$
have for dimensions $5$, $10$ and $19$. There remains to study the cases
of $[1,2]_4$ and $[1,3]_5$, since $\dim [2,2]_5 = 6 > 4$. 
For $\la = [1,2]_4$ we have $\dim \la = 3$ and the case is similar to the
one of $[0,1]_3$ above. For $\la = [1,3]_5$ we have $\dim \la = 4$ and the
case is similar to the one of $[1,0]_5$ above. This concludes these cases.

\medskip


Since the case $k = 0$ has been done, we can assume $k \geq 1$. If $a = r-1$,
this implies $k = 0$, so we can also assume $a \leq r-2$. If $a = r-2$, then
$k \leq 1$ hence $k = 1$, and this is the case of the $[1,r-2]_r$, which is also done.
Thus $a \leq r-3$, $k \geq 1$. 

Assume for a while $a = 0$. Then $k \leq r-1$ ; the 1-dimensional
case $[0,r-1]_r$ is trivial hence we can assume $k \leq r-2  = r-a-2$. Then
$\Res \la$ has 3 simple components (1), (2) and (3), hence we are done if
we know that $\dim \la > 6$. For $r \geq 6$, it is easily checked
that $\Res_5 [0,k]_5$ contains one of the $[0,l]_5$ for $1 \leq l \leq 3$,
which have dimension $> 6$. The remaining cases are $[0,1]_3$, $[0,1]_4$, $[0,2]_4$,
which have already been tackled.

We can now assume $1 \leq a \leq r-3$ and $k \geq 1$. Since the case
$k = r-a-1$ has been tackled we can assume $k \leq r-a-2$. In that case,
$\Res \la$ admits the 4 components (1)-(4), so we are done if
$\dim \la > 8$. Note that, under our conditions
$1 \leq a \leq r-3$ and $1 \leq k \leq r-a-2$, for $r \geq 5$,
at least one of the components (2),(3) and (4) fulfills
the same conditions for $r-1$. Since $\dim [1,1]_5 = \dim [1,2]_5 = 20$
and $\dim [2,1]_5 = 15$, this shows by induction that $\dim \la > 8$ for $r \geq 5$.
The only remaining case is for $\la = [1,1]_4$. We have $\rk \h = 4$ hence
$\rk \g \geq 4$. If $\rk \g > 4 = (\dim \la)/2$ we are done, so we
need to exclude the case $\rk \g = 4$. By lemma 3.3 (I) of \cite{IH2} we
know that $\g$ is simple, and by lemma 3.4 of \cite{IH2} the fact
that $\la$ is not selfdual provides a contradiction. This concludes the proof
of the theorem.
\end{proof}

Theorem \ref{theoimlie} implies the `Zariski-closure' part of theorem \ref{theoimB}
of the introduction, as in \cite{BMW} for the BMW-algebra
and \cite{LIETRANSP} for the Hecke algebra. The remaining part of
the theorem is proved in the section below.

\subsection{Faithfulness of $B_n \to LG_n(\alpha)^{\times}$}


Here we assume that $K$ is a field of characteristic $0$ and
$a,b,c \in K^{\times}$ three algebraically independent elements.
For convenience we moreover assume that $-1$, $a$ and $b$ admit square roots in $K$.
We let $S_n$ denote the $1$-dimensional $K B_n$ module defined by $s_i \mapsto a$,
$U_n$ denote the $(n-1)$-dimensional $K B_n$ module afforded by the
reduced Burau representation (convention : the image of $s_i$ has eigenvalues
$a$ with multiplicity $(n-2)$
and
$b$ with multiplicity $1$). Recall that the Krammer representation $Kr_n$
defined in \cite{KRAMMER4}
is a $n(n-1)/2$-dimensional irreducible representation 
of $B_n$ over $\Q(q,t)$, such that the image of $s_i$ has 3 eigenvalues $1, -q, tq^2$. This provides a faithful representation of $B_n$.
Up to renormalization ($s_i \mapsto \la s_i$) and change of parameter we can assume instead that
the three eigenvalues are $a,b,c$ and that this representation is defined over $K$, without
affecting the faithfulness property.

This representation factors through the BMW-algebra. Usually this algebra is defined over the field $\Q(s,\alpha)$ of rational fractions
(see e.g. \cite{BMW}), however under the same renormalization process $s_i \mapsto \la s_i$ one can define
it over $K$ in such a way that the image of $s_1$ has eigenvalues $a,b,c
$
(explicitely, $a = \la s, b = -\la s^{-1}, c = -\la \alpha^{-1}$, and conversely
$\la = \sqrt{-ab}, s = \sqrt{-a/b}, \alpha = - c^{-1}\sqrt{-ab} $, whence our assumptions on $K$).
The algebra $BMW_n$ is semisimple, with irreducible representations parametrized by partitions of $m$
for $0 \leq m \leq n$ and $n-m$ an even integer.  For $n = 2$, the empty partition $\emptyset$ of $m=0$  corresponds to $s_1 \mapsto c$
while $[2]$ and $[1,1]$ correspond to $s_1 \mapsto a$ and  $s_1 \mapsto b$, respectively. Under this
convention, $Kr_n$ is the irreducible component labelled by the partition $[n-2]$, and the reduced Burau representation
is labelled by $[n-1,1]$.


It thus has the property that
$Kr_2$ is $s_1 \mapsto c$
and the restriction of $Kr_{n+1}$ to $B_n \subset  B_{n+1}$ is
$S_{n} + U_n + Kr_n$. In diagrammatic terms, $Kr_n$ has a Bratteli diagram of the form
$$
\xymatrix{
Kr_n \ar@{-}[d]  \ar@{-}[dr] \ar@{-}[drr] \\
Kr_{n-1} \ar@{-}[d]  \ar@{-}[dr] \ar@{-}[drr] & U_{n-1} \ar@{-}[d]  \ar@{-}[dr]  & S_{n-1} \ar@{-}[d] \\
Kr_{n-2} \ar@{-}[d]  \ar@{-}[dr] \ar@{-}[drr] & U_{n-2} \ar@{-}[d]  \ar@{-}[dr]  & S_{n-2} \ar@{-}[d] \\
\dots & \dots & \dots
}
$$

We prove that this property characterizes the Krammer representation.

\begin{prop} Let $Kr_n$ for $n \geq 2$ denote a family of irreducible $K B_n$-modules, with the
property that the restriction of $Kr_{n+1}$ to $K B_n \subset K B_{n+1}$ is
$S_{n} + U_n + Kr_n$, and that $Kr_2$ is the $1$-dimensional $K B_2$-module $s_i \mapsto c$. Then $Kr_n$ is isomorphic to the Krammer representation.
\end{prop} 
\begin{proof}
When $n = 2$ there is nothing to prove, and for $n =3$ we know from the description of $H_3$ that there is up to isomorphism
only one irreducible 3-dimensional representation of $B_3$ where $s_1$ has 3 distinct eigenvalues $a,b,c$, so we
can assume $n \geq 4$. Then the restriction of $K_n$ to $B_3$ is a direct sum of irreductible representations which factorize through the Birman-Wenzl-Murakami algebra (up to a change of parameters and renormalization),
and it thus needs to factorize through the BMW-algebra ; since the relations for the BMW-algebra are generated by relations in $B_3$
this proves that $K_n$ itself factorizes through the BMW-algebra. It is then a simple combinatorial task to check that the only irreducible representations of the BMW-algebra for $n \geq 4$ with this Bratteli diagram corresponds to the partition $[n-2]$.
\end{proof}

\begin{cor} For $n \geq 2$, the Krammer representation factorizes through $LG_n$.
\end{cor}

\begin{proof}
The $K B_n$-module $V_{[0,n-2]_n}$
obviously satisfies the assumption, and factors through $LG_n$.
\end{proof}

\begin{cor} For $n \geq 2$, the morphism $B_n \to LG_n^{\times}$ is into.
\end{cor}
\begin{proof} Immediate consequence of the faithfulness of the Krammer representation.
\end{proof}


\bigskip



\end{document}

%% file: cl.pdf_t
\begin{picture}(0,0)%
\includegraphics{cl.pdf}%
\end{picture}%
\setlength{\unitlength}{592sp}%
\begingroup\makeatletter\ifx\SetFigFont\undefined%
\gdef\SetFigFont#1#2#3#4#5{%
  \reset@font\fontsize{#1}{#2pt}%
  \fontfamily{#3}\fontseries{#4}\fontshape{#5}%
  \selectfont}%
\fi\endgroup%
\begin{picture}(14123,12066)(-13189,-11194)
\put(-3599,-5161){\makebox(0,0)[lb]{\smash{{\SetFigFont{9}{10.8}{\familydefault}{\mddefault}{\updefault}$B_n$}}}}
\put(-12599,-5161){\makebox(0,0)[lb]{\smash{{\SetFigFont{9}{10.8}{\familydefault}{\mddefault}{\updefault}$B_n$}}}}
\end{picture}%

%% file: closure.pdf_t
\begin{picture}(0,0)%
\includegraphics{closure.pdf}%
\end{picture}%
\setlength{\unitlength}{592sp}%
\begingroup\makeatletter\ifx\SetFigFont\undefined%
\gdef\SetFigFont#1#2#3#4#5{%
  \reset@font\fontsize{#1}{#2pt}%
  \fontfamily{#3}\fontseries{#4}\fontshape{#5}%
  \selectfont}%
\fi\endgroup%
\begin{picture}(23348,12741)(-4189,-11869)
\put(10951,-7561){\makebox(0,0)[lb]{\smash{{\SetFigFont{9}{10.8}{\familydefault}{\mddefault}{\updefault}$\alpha$}}}}
\put(-3449,-3361){\makebox(0,0)[lb]{\smash{{\SetFigFont{9}{10.8}{\familydefault}{\mddefault}{\updefault}$\alpha$}}}}
\put(-3449,-7561){\makebox(0,0)[lb]{\smash{{\SetFigFont{9}{10.8}{\familydefault}{\mddefault}{\updefault}$\beta$}}}}
\put(3751,-3361){\makebox(0,0)[lb]{\smash{{\SetFigFont{9}{10.8}{\familydefault}{\mddefault}{\updefault}$\beta$}}}}
\put(3751,-7561){\makebox(0,0)[lb]{\smash{{\SetFigFont{9}{10.8}{\familydefault}{\mddefault}{\updefault}$\alpha$}}}}
\put(10951,-3361){\makebox(0,0)[lb]{\smash{{\SetFigFont{9}{10.8}{\familydefault}{\mddefault}{\updefault}$\beta$}}}}
\end{picture}%

%% file: inj.pdf_t
\begin{picture}(0,0)%
\includegraphics{inj.pdf}%
\end{picture}%
\setlength{\unitlength}{592sp}%
\begingroup\makeatletter\ifx\SetFigFont\undefined%
\gdef\SetFigFont#1#2#3#4#5{%
  \reset@font\fontsize{#1}{#2pt}%
  \fontfamily{#3}\fontseries{#4}\fontshape{#5}%
  \selectfont}%
\fi\endgroup%
\begin{picture}(4555,5466)(5486,-5794)
\put(7576,-3361){\makebox(0,0)[lb]{\smash{{\SetFigFont{9}{10.8}{\familydefault}{\mddefault}{\updefault}$B_{n-k}$}}}}
\put(5626,-3361){\makebox(0,0)[lb]{\smash{{\SetFigFont{9}{10.8}{\familydefault}{\mddefault}{\updefault}$B_k$}}}}
\end{picture}%